\documentclass[a4paper,12pt,final]{amsart}

\title{Fishing for complements}
\author{Lidia Angeleri H\"ugel, David Pauksztello and Jorge Vit\'oria}

\usepackage{amssymb}
\usepackage{latexsym}
\usepackage{amsmath}
\usepackage{mathrsfs}
\usepackage{mathtools}
\usepackage{enumitem}
\usepackage{graphicx}
\usepackage{fullpage}
\usepackage{pdflscape}

\usepackage{adjustbox}

\usepackage[T1]{fontenc}

\usepackage{tikz}
\usetikzlibrary{decorations.pathmorphing}
\usetikzlibrary{decorations.pathreplacing}
\usetikzlibrary{positioning}
\usetikzlibrary{calc}
\usetikzlibrary{backgrounds}
\usetikzlibrary{shapes.misc}
\usetikzlibrary{patterns}
\usetikzlibrary{math}

\usepackage{tikz-cd}
\tikzcdset{scale cd/.style={every label/.append style={scale=#1},
    cells={nodes={scale=#1}}}}
    
\usepackage{xcolor}

\definecolor{ForestGreen}{RGB}{34,139,34}

\usepackage{hyperref}
\hypersetup{
    colorlinks,
    linkcolor={red!50!black},
    citecolor={blue!50!black},
    urlcolor={blue!80!black}
}

\newcommand{\harxiv}[1]{\href{http://arxiv.org/abs/#1}{\texttt{arXiv:#1}}}
\newcommand{\hyref}[2]{ \hyperref[#2]{#1~\ref*{#2}} }

\theoremstyle{plain}
\newtheorem{theorem}{Theorem}[section]
\newtheorem{lemma}[theorem]{Lemma}
\newtheorem{corollary}[theorem]{Corollary}
\newtheorem{proposition}[theorem]{Proposition}
\newtheorem{introtheorem}{Theorem}

\theoremstyle{definition}
\newtheorem{remark}[theorem]{Remark}
\newtheorem{example}[theorem]{Example}
\newtheorem{definition}[theorem]{Definition} % numbered definition.

\newcommand{\sA}{\mathsf{A}}
\newcommand{\sC}{\mathsf{C}}
\newcommand{\sD}{\mathsf{D}}
\newcommand{\sH}{\mathsf{H}}
\newcommand{\sM}{\mathsf{M}}
\newcommand{\sN}{\mathsf{N}}
\newcommand{\sT}{\mathsf{T}}
\newcommand{\sU}{\mathsf{U}}
\newcommand{\sV}{\mathsf{V}}
\newcommand{\sW}{\mathsf{W}}
\newcommand{\sX}{\mathsf{X}}
\newcommand{\sY}{\mathsf{Y}}

\newcommand{\sh}{\mathsf{h}}
\newcommand{\sv}{\mathsf{v}}
\newcommand{\sw}{\mathsf{w}}

\newcommand{\Db}{\mathsf{D}^b}
\newcommand{\Kb}{\mathsf{K}^b}
\newcommand{\per}[1]{\mathsf{per}(#1)}

\newcommand{\cD}{\mathcal{D}}
\newcommand{\cF}{\mathcal{F}}
\newcommand{\cS}{\mathcal{S}}
\newcommand{\cT}{\mathcal{T}}

\newcommand{\bZ}{\mathbb{Z}}

\renewcommand{\geq}{\geqslant}
\renewcommand{\leq}{\leqslant}
\renewcommand{\phi}{\varphi}

\DeclareMathOperator{\End}{\mathsf{End}}
\DeclareMathOperator{\Hom}{\mathsf{Hom}}
\DeclareMathOperator{\Ext}{\mathsf{Ext}}

\newcommand{\add}[1]{\mathsf{add}(#1)}
\newcommand{\Add}[1]{\mathsf{Add}(#1)}
\renewcommand{\mod}[1]{\mathsf{mod}(#1)}
\newcommand{\proj}[1]{\mathsf{proj}(#1)}
\newcommand{\Proj}[1]{\mathsf{Proj}(#1)}
\newcommand{\Mod}[1]{\mathsf{Mod}(#1)}
\newcommand{\Gen}[1]{\mathsf{Gen}(#1)}

\newcommand{\kk}{{\mathbf{k}}}

\newcommand{\orth}{^\perp}

\DeclareMathOperator{\susp}{\mathsf{susp}}
\DeclareMathOperator{\Susp}{\mathsf{Susp}}
\DeclareMathOperator{\cosusp}{\mathsf{cosusp}}
\newcommand{\thick}[2]{\mathsf{thick}_{#1}(#2)}

\newcommand{\SiltL}[2]{\mathsf{Silt}^{#1}_{#2}(\Lambda)}
\newcommand{\siltL}[2]{\mathsf{silt}^{#1}_{#2}(\Lambda)}
\newcommand{\SiltLambda}{\mathsf{Silt}(\Lambda)}
\newcommand{\siltLambda}{\mathsf{silt}(\Lambda)}

\newcommand{\too}{\longrightarrow}
\newcommand{\rightlabel}[1]{\stackrel{#1}{\longrightarrow}}
\newcommand{\bij}{\stackrel{1-1}{\longleftrightarrow}}

\newcommand{\Stovicek}{\v{S}\v{t}ov\'{i}\v{c}ek}

\makeatletter
\@namedef{subjclassname@2020}{%
  \textup{2020} Mathematics Subject Classification}
\makeatother

\newcommand{\Addresses}{{% additional braces for segregating \footnotesize
  \bigskip
  \footnotesize
L.~Angeleri H\"ugel, \textsc{Dipartimento di Informatica - Settore di Matematica, Universit\`a degli Studi di Verona, Strada le Grazie 15 - Ca' Vignal, I-37134 Verona, Italy}\par\nopagebreak
  \textit{E-mail address:} \texttt{lidia.angeleri@univr.it}

  \medskip

D.~Pauksztello, \textsc{School of Mathematical Sciences, Lancaster University, Lancaster, LA1 4YF, United Kingdom}\par\nopagebreak
  \textit{E-mail address:} \texttt{d.pauksztello@lancaster.ac.uk}

  \medskip

J.~Vit\'oria, \textsc{Dipartimento di Matematica ``Tullio Levi-Civita'', Universit\`a degli Studi di Padova, Torre Archimede, via Trieste 63, 35121 Padova, Italy}\par\nopagebreak
  \textit{E-mail address:} \texttt{jorge.vitoria@unipd.it}
}}

\begin{document}

\begin{abstract}
Given a presilting object in a triangulated category, we find necessary and sufficient conditions for the existence of a complement. This is done both for classic (pre)silting objects and for large (pre)silting objects. The key technique is the study of associated co-t-structures. As a consequence of our techniques we recover some known cases of the existence of complements, including for derived categories of some hereditary abelian categories and for silting-discrete algebras. Moreover, we also show that a finite-dimensional algebra is silting discrete if and only if every bounded large silting complex is equivalent to a compact one.
\end{abstract}


%05E10: Combinatorial aspects of representation theory
%16G10: Representations of Artinian rings
%18G05: Projectives and injectives (category theoretic aspects)
%18G80: Derived categories, triangulated categories

\maketitle

%============================================================================
% Introduction
\section{Introduction} 

The question about the existence of complements is a problem that goes back to the early days of tilting theory. Bongartz showed in 1981 that a partial tilting module over a finite-dimensional algebra $\Lambda$ always admits a finite-dimensional complement \cite{Bo}. Here a partial tilting module is a finite-dimensional module without self-extensions that has projective dimension one.   Already for projective dimension two, however, there are counterexamples to the corresponding generalised statement \cite{RS}.  On the other hand, it was shown in 
\cite{AC02} that complements do exist for  partial tilting modules of any projective dimension if we relax the requirement that they ought to be finite dimensional, working with \textit{large} tilting modules, i.e.~possibly infinite-dimensional tilting modules, instead. This result relies on the theory of cotorsion pairs developed in \cite{ET}, which is a source of left and right approximations in the category $\Mod{\Lambda}$ of all $\Lambda$-modules, with good homological behaviour. 

The analogous problem in silting theory asks whether a presilting object can be completed to a silting object. One has to distinguish between two parallel setups: the classic notion of (pre)silting object from  \cite{AI12,KV88} which is mostly used in triangulated categories satisfying some finiteness conditions, and the more recent definition from \cite{NSZ19,PV18} designed for `large' triangulated categories with arbitrary coproducts. While in the references indicated, these subcategories are simply called (pre)silting, in this paper we will use the adjectives \textit{classic} and \textit{large} to distinguish them.

Bongartz completion extends to silting theory,  see \cite[\S 5]{DeFe15}, \cite[Proposition 6.1]{Wei13}, \cite[Proposition 3.14]{BY}. The silting version is  `basis-free': the assumption that a partial 1-tilting module has a two-term projective resolution is replaced by the condition that the presilting object $X$ is `two-term' with respect to a suitable classic silting object $M$.
Moreover, in the classic setup, complements are known to exist in certain ambient triangulated categories, such as the bounded derived category of a hereditary abelian category, or the category $\per{\Lambda}$ of perfect complexes when $\Lambda$ is a piecewise hereditary algebra or a silting-discrete algebra \cite{BY,DaFu22,AM17}. On the other hand, recent work in \cite{LZ,JSW,K} provides examples of finite-dimensional algebras for which $\per{\Lambda}$ contains classic presilting objects that cannot be completed to a classic silting object.

In analogy to  \cite{AC02}, we show  in the present paper that every classic presilting object in $\per{\Lambda}$  can be completed to a bounded complex of (possibly large) projectives which is a large silting object in the unbounded derived category $\sD(\Mod{\Lambda})$. This  relies  on a development of the  theory of cotorsion pairs in \cite{SS11}  providing a powerful existence result for co-t-structures which we state in Lemma~\ref{saorin-stovicek}. 

In fact, we present general criteria for the existence of complements, both in the context of large and classic silting theory (Theorems~\ref{large completion} and \ref{small completion}). Under mild assumptions on the ambient triangulated category $\sT$, we prove that
a presilting object $X$  admits a complement  if and only if there exists a silting object $M$ satisfying the following conditions, where we denote by $\sX$ and $\sM$ the additive closures of $X$ and $M$ in $\sT$, respectively:

\begin{enumerate}[label=(\roman*)]
\item \label{intermediate} $X$ is intermediate with respect to $M$, i.e. $\sM\geq \sX\geq \sM[n]$ for some  $n>0$; 
\item \label{averaging}  $\sM\geq \sX$ and given the co-t-structures $(\sU_X,\sV_X:=(X[<0])\orth)$ and $(\sU_M,\sV_M:=(M[<0])\orth)$ associated to $X$ and $M$, respectively, the intersection $\sV_X \cap \sV_M $ is again the coaisle of a co-t-structure in $\sT$.
\end{enumerate}

Condition \ref{intermediate} is natural in the classic case, because the existence of a classic silting object entails that   homomorphisms between two objects vanish after shifting one object far enough (Lemma~\ref{only finite}).
Condition \ref{averaging} is the co-t-structure analogue of averaging of t-structures studied in \cite{BPP13}, and it is always satisfied in the large setup thanks to  Lemma~\ref{saorin-stovicek}.
In the classic setup, the choice of the silting object $M$ above matters. Indeed, the averaging condition \ref{averaging} may hold for certain objects $M$ and fail for others, see  Example~\ref{A2}. We summarise our main theorem as follows.
 
\begin{introtheorem}
Let $\sT$ be a triangulated category. The following hold.
\begin{enumerate}
\item {\rm (Theorem~\ref{large completion})} If $\sT$ is algebraic and compactly generated and if $X$ is a large presilting object in $\sT$, then $X$ admits a complement to a large silting object if and only if condition \ref{intermediate} above holds with respect to a large silting object $M$.
\item {\rm (Theorem~\ref{small completion})} If $\add{T}$ is precovering in $\sT$ for every $T$ in $\sT$ (for example, if $\sT$ is $\kk$-linear, Krull-Schmidt and Hom-finite over a field $\kk$) and $X$ is a classic presilting object in $\sT$, then $X$ admits a complement to a classic silting object if and only if condition \ref{averaging} above holds with respect to a classic silting object $M$.
\end{enumerate}
\end{introtheorem}

In the last part of the paper we will show how to recover Bongartz completion and the existence of complements for classic presilting objects over hereditary abelian categories or silting-discrete algebras. The latter result also requires  a characterisation of silting-discrete algebras that may be of independent interest.

\begin{introtheorem}[Theorem~\ref{large is small silting discrete}] 
A finite-dimensional algebra $\Lambda$ is silting discrete if and only if every bounded complex of projective modules which is a large silting object in  $\sD(\Mod{\Lambda})$ is (additively) equivalent to a classic silting object in $\per{\Lambda}$.
\end{introtheorem}
 
This can be regarded as a triangulated version of a result from \cite{AMV19} stating that  $\Lambda$ is $\tau$-tilting finite  if and only if every silting module in $\Mod{\Lambda}$ is equivalent to a support $\tau$-tilting module in  $\mod{\Lambda}$.

The article is organised as follows. In Sections~\ref{sec:prelim} and \ref{sec:co-t-structures} we collect some preliminaries and review the notions of  silting or presilting objects and subcategories, together with their relationship with co-t-structures. Section~\ref{sec:complements} contains the general existence results for complements, while Section~\ref{sec:applications} recovers some known cases as applications of our criteria. Section~\ref{sec:silting-discrete} is devoted to silting-discrete algebras.

%============================================================================
% SECTION
\section{Preliminaries} \label{sec:prelim}
%============================================================================

In this section we fix some notation and terminology. Unless stated otherwise, $\sT$ will denote an abstract triangulated category with shift functor $[1]$, and all subcategories will be strict and full. Furthermore, when considering abelian categories, we shall consider only those whose derived category exists, i.e.\ we require that morphisms between any two given objects form a set rather than a proper class.

\subsection{Subcategory constructions} \label{sub:subcategory-constructions}
For subcategories $\sU$, $\sV$ and $\sX$ of $\sT$, 
we consider the following subcategories of $\sT$:

\noindent
\begin{tabular}{@{} p{0.13\textwidth} @{} p{0.87\textwidth} @{}}
$\sU * \sV$     & the  subcategory of $\sT$ 
consisting of objects $T \in \sT$ for which there is a triangle $U \to T \to V \to U[1]$ with $U \in \sU$ and $V \in \sV$. 
                         If $\sU * \sU = \sU$ then $\sU$ is said to be \emph{closed under extensions} or \emph{extension-closed}. \\

$\thick{}{\sX}$ & the \emph{thick subcategory generated by $\sX$}, the smallest thick (i.e.\ triangulated and closed under direct summands) subcategory of $\sT$ containing $\sX$. \\
$\susp(\sX)$ & the \emph{suspended subcategory generated by $\sX$}, the smallest  subcategory of $\sT$ containing $\sX$ which is closed under suspensions, extensions and direct summands.\\
$\cosusp(\sX)$  & the \emph{cosuspended subcategory generated by $\sX$}, the smallest  subcategory of $\sT$ containing $\sX$  closed under cosuspensions, extensions and direct summands. \\
$\Susp(\sX)$ & the  smallest  subcategory of $\sT$ containing $\sX$ which is closed under suspensions, extensions and existing coproducts (and thus also  under direct summands). \\

$\add{\sX}$     & the \emph{additive closure of $\sX$ in} $\sT$ formed by  all summands of finite coproducts of objects in $\sX$ (which exist in $\sT$ since it is an additive category). \\       
\end{tabular}

\noindent\begin{tabular}{@{} p{0.13\textwidth} @{} p{0.87\textwidth} @{}}

$\Add{\sX}$     & the \emph{large additive closure of $\sX$ in} $\sT$ given by  all summands of existing coproducts of objects in $\sX$. 
\\ 

$\sX\orth$      & the \emph{right orthogonal} to $\sX$,  given by the objects $T \in \sT$ with $\Hom(X,T)=0$ for each $X \in \sX$.
For a set of integers $I$ (often expressed by symbols such as $>n$, $<n$, $\geq n$, $\leq n$  with the obvious associated meaning), we write $\sX[I]\orth$  for the  subcategory formed by the objects $T\in \sT$ with $\Hom(X[i],T) = 0$ for each $X \in \sX$ and $i \in I$.
                           \\

${}\orth\sX$   & the \emph{left orthogonal} to $\sX$,  given by the objects $T \in \sT$ with $\Hom(T,X)=0$ for each $X \in \sX$.
                          The subcategory ${}\orth(\sX[I])$,  $I\subseteq \bZ$,  is defined analogously as above. \\
                                                              
\end{tabular}

We will use the following abbreviations: 
\[
\sV_\sX= \sX[<0]\orth, \quad \sW_\sX= \sX[\geq 0]\orth, \quad \sU_\sX={}\orth(\sV_\sX).
\]
In the notation of \cite{AMV16,AMV19,AMV20,PV18} we have  $\sV_\sX=\sX^{\perp_{>0}}$,\, $\sW_\sX= \sX^{\perp_{\leq 0}}$,  $\sU_\sX={}^{\perp_0}(\sV_\sX)$. Note that $\sV_\sX$ is a suspended subcategory, while $\sW_\sX$ and $\sU_\sX$ are cosuspended subcategories. When $\sX$ consists of a single object $X$, we just write $\thick{}{X},\Add{X}, \sV_X,\sW_X$ etc. 

We say that a subcategory $\sX$ of $\sT$ \emph{weakly generates} $\sT$ if $(\sX[\bZ])\orth=0$. The adverb ``weakly'' here refers to the fact that if $\thick{}{\sX} = \sT$ then $\sX$ weakly generates $\sT$, while the converse does not hold in general. 

\subsection{Precovering and preenveloping subcategories} \label{sec:approx}

Let $\sU$ be a subcategory of $\sT$ and $T$ be an object of $\sT$. A morphism $f \colon U_T \to T$ is called a \emph{$\sU$-precover} (or a \emph{right $\sU$-approximation}) of $T$ if the induced homomorphism 
\[
\Hom_\sT(U,f) \colon \Hom_\sT(U,U_T) \to \Hom_\sT(U,T)
\]
is surjective for each object $U$ of $\sU$.
If each object $T$ of $\sT$ admits a $\sU$-precover then $\sU$ is said to be \emph{precovering} in $\sT$.
There are dual notions of \emph{$\sU$-preenvelope} and \emph{preenveloping} subcategory.

\subsection{t-structures and co-t-structures} \label{sec:torsion-pairs}

Two kinds of torsion pairs in triangulated categories  play an important role in silting theory.

\begin{definition}[{\cite[Definition 2.2]{IYo08}}]
A pair $(\sU,\sV)$ of idempotent-complete additive subcategories of a triangulated category $\sT$ is said to be a \emph{torsion pair} in $\sT$ if
\begin{enumerate}[label=(\arabic*)]
\item $\Hom_\sT(U,V)=0$, for each $U$ in $\sU$ and $V$ in $\sV$;
\item $\sT = \sU * \sV$.
\end{enumerate}
For each object $T$ of $\sT$, the triangle associated with the decomposition $\sT = \sU * \sV$ is called the \emph{truncation triangle} for $T$.
The subcategories $\sU$ and $\sV$ are the \emph{aisle} and \emph{coaisle}, respectively, of the torsion pair.
A torsion pair $(\sU,\sV)$ is said to be:
\begin{itemize}
\item a \emph{t-structure} if $\sU[1] \subseteq \sU$ (see \cite{BBDG});
\item a \emph{co-t-structure} if $\sU[-1]\subseteq \sU$ (see \cite{Pauk08}, or \cite{Bondarko10} under the name \emph{weight structure});
\item \emph{bounded} if $\sT = \bigcup_{n \in \bZ} \sU[n] = \bigcup_{n \in \bZ}\sV[n]$;
\item \emph{left nondegenerate} if $\bigcap_{n\in\bZ}\sU[n]=0$;
\item \emph{right nondegenerate} if $\bigcap_{n\in\bZ}\sV[n]=0$;
\item \emph{nondegenerate} if it is both left and right nondegenerate;
\item \emph{generated by a set} if there is a set of objects $\sX$ in $\sT$ such that $\sV=\sX\orth$.
\end{itemize}
After \cite{Bondarko10}, given torsion pairs $(\sU,\sV)$ and $(\sV,\sW)$, we say that the former is \emph{left adjacent} to the latter or that the latter is \emph{right adjacent} to the former.

In the case that $(\sU,\sV)$ is a t-structure, the subcategory $\sA = \sU \cap \sV[1]$ called the \emph{heart}. In the case that $(\sU,\sV)$ is a co-t-structure, the subcategory $\sC = \sU[1] \cap \sV$ is called the \emph{coheart} of the co-t-structure. A t-structure $(\sU,\sV)$ is called \emph{split} if each truncation triangle given in condition $(2)$ above is a split triangle, i.e.\ $\sT = \Add{\sU,\sV}$. 
\end{definition}

We recall a few useful results about t-structures and co-t-structures:
\begin{enumerate}
\item The aisle of a torsion pair is always a precovering subcategory and the coaisle is always a preenveloping subcategory.
\item The heart of a t-structure is an abelian category (\cite{BBDG}). The coheart of a co-t-structure is an additive subcategory, but rarely abelian.
\item The truncation triangles for a t-structure are functorially determined.
\item A t-structure with heart $\sA$ is bounded if and only if 
\[
\sT = \bigcup_{i \geq j} \sA[i] * \sA[i-1] * \cdots * \sA[j].
\]
\end{enumerate}

We end this section with a couple of straightforward but useful observations about co-t-structures.
For the first one, see also \cite[Proposition 1.5.6]{Bondarko10}.

\begin{lemma} \label{lem:Postnikov}
Let $(\sU, \sV)$ be a co-t-structure in $\sT$.
Then any object of $\sT$ sits in a (possibly infinite) Postnikov tower, that is, a diagram of the form
\[
\begin{tikzcd}%[sep=\scriptsize]
\cdots \ar{r} & T_3 \ar{r} \ar{d}        & T_2 \ar{r} \ar{d}       &  T_1 \ar{r} \ar[d]                & T_0 = T \ar{d} \\
              & C_3[-3] \ar[dashed]{ul}  & C_2[-2] \ar[dashed]{ul} &  C_1[-1] \ar[dashed]{ul} & V_T \ar[dashed]{ul}
\end{tikzcd}
\]
where
\[
T_1\to T\to V_T\to T_1[1]
\quad \text{and} \quad
T_{i+1}\to T_{i}\to C_i[-i]\to T_{i+1}[1]
\]
are triangles for all $i\geq 1$, and such that $V_T$ lies in $\sV$, $T_i$ lies in $\sU[1-i]$ and $C_i$ in the coheart $\sC = \sU [1] \cap \sV$ for all $i\geq 1$.
Moreover, if $T$ lies in $\sV[-n]$ for some $n >0$, then the Postnikov tower is finite, i.e. $T_n$ lies in $\sC[-n]$. 
\end{lemma}

As truncations with respect to co-t-structures are not unique, the Postnikov tower occurring in Lemma~\ref{lem:Postnikov} is not necessarily unique.

\begin{proof}
The existence of the Postnikov tower follows by iteratively taking triangles, starting with $T_0 = T$, and  choosing the co-t-structure $(\sU[-i],\sV[-i])$ for each $T_i$, $i\geq 0$.
It is then easy to observe that the third term in each truncation triangle sits in the subcategories claimed.

Suppose now that $T$ lies in $\sV[-n]$ for some $n>0$. One can show by induction that $T_i \in \sV[-n]$ for each $0 \leq i \leq n$. Indeed, suppose that, for $i\geq 1$, $T_{i-1}$ lies in $\sV[-n]$. One then reads off that $T_i \in \sV[-n]$ from the truncation triangle
\[
C_{i-1}[-i] \to T_i \to T_{i-1} \to C_{i-1}[-i+1]
\]
and the fact that $C_{i-1} [-i] \in \sC[-i] \subseteq \sV[-i] \subseteq \sV[-n]$. Hence, by construction of the tower, $T_n$ lies in $\sV[-n]\cap \sU[-n+1]=\sC[-n]$.
\end{proof}

\begin{corollary}\label{intermed}
Let $(\sU, \sV)$ and $(\sU', \sV')$ be two co-t-structures in $\sT$ with cohearts $\sC$ and $\sC'$, respectively. Assume there is an integer $n>0$ such that $\sV'[n]\subseteq \sV\subseteq\sV'$. Then 
\[
\sC'\subseteq \sC[-n] * \cdots * \sC[-1] * \sC
\quad  \text{and} \quad 
\sC \subseteq \sC' * \cdots * \sC'[n-1] * \sC'[n].
\] 
In particular, it follows that $\thick{}{\sC'}=\thick{}{\sC}.$
\end{corollary}

\begin{proof}
First of all, notice that the assumption also yields $\sU'[n] \supseteq \sU \supseteq \sU'$. Take now $M$ in $\sC'=\sU'[1]\cap\sV'$ and consider a triangle 
\[
U\to M\to V\to U[1]
\] 
with $U$ in $\sU$ and $V$ in $\sV$. 
Since $M$ lies in $\sV'\subseteq\sV[-n]$, by Lemma \ref{lem:Postnikov} there is a finite Postnikov tower showing that the object $U$ in the triangle  lies in $\sC[-n] * \cdots * \sC[-1]$. 
Observe further that $M$ lies in $\sU'[1]\subseteq\sU[1]$, hence $V$ lies in $\sC$. We conclude that $M$ lies in $\sC[-n] * \cdots * \sC[-1] * \sC$, as desired. 

For the second inclusion, pick $T$ in $\sC = \sU[1] \cap \sV \subseteq \sU'[n+1]\cap \sV'$
and apply Lemma~\ref{lem:Postnikov} on the object $T[-(n+1)]$. 
\end{proof}

\subsection{Derived categories} \label{sec:derived-categories}

Our main examples come from various categories associated to a coherent ring $R$. Later in Section~\ref{sec:silting-discrete} we will restrict to the case of a finite-dimensional algebra over a field $\kk$; for emphasis in this case we will denote the algebra by $\Lambda$.

\medskip
\noindent
\begin{tabular}{@{} p{0.15\textwidth} @{} p{0.85\textwidth} @{}}
$\Mod{R}$    & the category of right $R$-modules; \\
$\mod{R}$    & the subcategory of $\Mod{R}$ formed by the finitely presented $R$-modules; \\
$\Proj{R}$ & the subcategory of $\Mod{R}$ formed by projective modules; \\
$\proj{R}$ & the subcategory of $\Mod{R}$ formed by finitely generated projective modules;\\
$\sD(R)$   & the derived category $\sD(\Mod{R})$ of $\Mod{R}$;\\
$\Db(R)$  & the bounded derived category $\Db(\mod{R})$  of $\mod{R}$;\\
$\Kb(\Proj{R})$  & the subcategory of $\sD({R})$ given by bounded complexes of  projective modules; \\
$\per{R}$ &  the subcategory of $\Db(R)$ formed by bounded complexes of finitely generated projective modules, also called \textit{perfect complexes}.
\end{tabular}

%============================================================================
% SECTION
\section{(Pre)silting and co-t-structures} \label{sec:co-t-structures}
%============================================================================

There are two kinds of (pre)silting subcategories/objects in common use, depending on the context in which one is working. There is the classic definition of (pre)silting subcategory/object, used in `small' triangulated categories (\cite{AI12,KV88}), and the more recent definition of a silting object, better adapted to `large' triangulated categories (\cite{AMV20,NSZ19,PV18}). We review these notions below.

\subsection{Classic silting subcategories}

\begin{definition}\label{def:small-silting}
Let $\sM=\add{\sM}$ be a subcategory of a triangulated category $\sT$. We say that $\sM$ is
\begin{itemize}
\item \emph{classic presilting} if $\Hom_\sT(M, M^\prime[>0]) = 0$ for any objects $M$ and $M^\prime$ of $\sM$;
\item \emph{classic silting} if it is classic presilting and $\sT = \thick{}{\sM}$. 
\end{itemize}
An object $M$ in $\sT$ is a \emph{classic (pre)silting object} if $\add{M}$ is a classic (pre)silting subcategory of $\sT$.
\end{definition}

A first fundamental fact about classic silting subcategories is their close relationship to co-t-structures.

\begin{theorem}[{\cite[Corollary 5.8]{MSSS13}}] \label{MSSS}
Let $\sT$ be a triangulated category. The assignment 
\[
\sM\mapsto(\sU_\sM,\sV_\sM)
\]
is a bijection between classic silting subcategories of $\sT$ and bounded co-t-structures in $\sT$. 
Moreover, if $\sM$ is a classic silting subcategory of $\sT$, the associated bounded co-t-structure has coheart $\sM$ and satisfies $\sV_\sM= \susp(\sM)$ and 
$\sU_\sM= \cosusp(\sM[-1])={}\orth(\sM[\geq0])$.
\end{theorem}

Note that there is a priori no condition imposed on a triangulated category where a classic silting subcategory lives. Nevertheless, the fact that $\sT = \thick{}{\sM}$ imposes that  if $\sM$ is skeletally small (for example, when $\sM=\add{M}$ for a silting object $M$), then so is $\sT$. 

\begin{example}
Let $R$ be a coherent ring. Then $R$ is a classic silting object in $\Kb(\proj{R})$ and $\mathsf{Proj}(R)$ is a classic silting subcategory of $\Kb(\mathsf{Proj}(R))$.
\end{example}

The existence of a silting subcategory does impose a condition on the behaviour of morphisms in the triangulated category.

\begin{lemma}[{\cite[Proposition 2.4]{AI12}}] \label{only finite}
Suppose $\sT$ is a triangulated category containing a classic silting subcategory. Then, for any two objects $X$ and $Y$ in $\sT$, there is $n>0$ such that $\Hom_{\sT}(X,Y[>n])=0$
\end{lemma} 

\subsection{Large silting objects}

The existence of precovers and preenvelopes, see \S\ref{sec:approx}, is central in silting theory. In the classic setting, the relevant precovers and preenvelopes exist under suitable finiteness conditions on the category. In large silting theory, the existence of precovers and preenvelopes is guaranteed provided we work with at most a set (rather than a proper class) of objects. As such, in the large setting we restrict our attention to large silting \emph{objects} rather than \emph{subcategories}.

In this subsection $\sT$ will be a triangulated category that admits all set-indexed coproducts.

\begin{definition}\label{def:large-silting}
An object $M$ of $\sT$ is called
\begin{itemize}
\item \emph{large presilting} if $\Hom_\sT(M, M[>0]) = 0$, and $\sV_M$ is coproduct closed;
\item \emph{large silting} if $(\sV_M, \sW_M)$ is a t-structure in $\sT$. 
\end{itemize}
Two (pre)silting objects $M$ and $M^\prime$ are said to be equivalent if $\Add{M}=\Add{M^\prime}$
\end{definition}

\begin{remark} \label{classic silting in thick}
We make the following observations regarding large (pre)silting objects.
\begin{enumerate}
\item It follows from the very definition of a t-structure that any large silting object is also a large presilting object.
\item Any large (pre)silting object $X$ gives rise to a classic silting subcategory $\mathsf{Add}(X)$ in $\thick{}{\Add{X}}$ because $\sV_X$ is closed under coproducts.
\end{enumerate}
\end{remark}

The definition of large silting is borrowed from \cite{PV18}, while the notion of large presilting is closely related to the notion of partial silting introduced in \cite{AMV20}, and in a wide range of categories they coincide, see Remark~\ref{large silting = large presilting + weak generation} below. The definition of partial silting from \cite{AMV20} has the additional requirement of the existence of a t-structure.

Similar to the classic silting case, for suitable triangulated categories, silting and presilting objects of $\sT$ are related to co-t-structures in $\sT$. 
One context in which this relationship is well understood is that of a compactly generated triangulated category.

\begin{definition} 
An object $X$ in a triangulated category $\sT$  with set-indexed coproducts is \emph{compact} if the functor $\Hom_\sT(X,-)$ commutes with set-indexed coproducts. We say that $\sT$ is \emph{compactly generated} if the subcategory of compact objects $\sT^c$ is skeletally small and weakly generates $\sT$.
\end{definition}

\begin{remark}\label{large silting = large presilting + weak generation}
Compactly generated triangulated categories are examples of a larger class of triangulated categories called \emph{well generated}. It follows from recent results of Neeman in \cite{Neeman} that our large presilting objects coincide with the partial silting objects of \cite{AMV20} in the wider context of well-generated triangulated categories.
Moreover,
Neeman's result in \cite{Neeman}  on the generation of t-structures has the further consequence that in a well-generated triangulated category, an object is large silting if and only if it is a large presilting object which weakly generates the category. This is observed in \cite{AMV20} after {\it loc.~cit.}~Lemma 3.3, as a consequence of \cite[Theorem 1(2)]{NSZ19}. 
Recently, the same relation between silting and presilting objects was extended to arbitrary triangulated categories with coproducts in \cite{Breaz}.
\end{remark}

\begin{example}
The condition that $\sV_M$ is closed under coproducts in Definition~\ref{def:large-silting} doesn't require the object $M$ to be compact. For an example, consider the large silting object $M$ given by the projective resolution of an infinite dimensional tilting module $T$ over a finite-dimensional algebra $\Lambda$ and observe that the associated aisle $\sV_M$ is closed under coproducts, because it consists of the complexes $Y$ in the aisle of the standard t-structure of $\Lambda$ whose zero cohomology $H^0(Y)$ is isomorphic to a quotient of a direct sum of copies of $T$ (see, for example, \cite[Remark 3.6]{B19}).
\end{example}

\begin{theorem}[{\cite[Proposition 3.8, Theorem 3.9 and Corollary 3.10]{AMV20}}, {\cite[Theorem 2]{NSZ19}}] \label{large correspondence}
Let $\sT$ be a compactly generated triangulated category.  The assignment
\[
M \mapsto (\sU_M, \sV_M)  
\]
gives a bijection between equivalence classes of large (pre)silting objects in $\sT$ and co-t-structures $(\sU,\sV)$ that are generated by a set and admit a right adjacent t-structure which is (right) nondegenerate.
If  $M$ is a large silting object of $\sT$, the associated  co-t-structure has coheart $\Add M$, and   $\sV_\sM= \Susp(M)$ is the smallest aisle of $\sT$ containing $M$.
\end{theorem}

\begin{remark}\label{comparison}
Let $\sT$ be a compactly generated triangulated category.  An object $M\in\sT^c$ is classic (pre)silting in $\sT^c$ if and only if it is large (pre)silting in $\sT$. Indeed, every object $M\in\sT^c$ generates a t-structure $(\Susp(M), \sW_M)$ in $\sT$ (see, for example, \cite[Theorem A.1]{AJS}), and if $M$ is classic presilting we have $\Susp(M)\subseteq \sV_M$. If $M$ is furthermore classic silting in $\sT^c$ (and, hence, a weak generator in $\sT$), a $\sW_M$-preenvelope of an object $X$ of $\sV_M$ must lie in $M[\mathbb{Z}]^\perp$, thus showing that $X$ must lie in $\Susp(M)$ and that $\Susp(M)=\sV_M$. Conversely, every large silting object $M$ is a weak generator in $\sT$, and 
from, e.g. \cite[Proposition 3.4.15]{Krause22}, it follows that $\thick{\sT^c}{M} = \thick{\sT}{M} = \sT^c$.
\end{remark}

%============================================================================
% SECTION
\section{Complements} \label{sec:complements}
%============================================================================

Our main problem in this paper is that of finding necessary and sufficient conditions for a given (classic/large) presilting object to be a summand of a silting object.
We will approach this problem by looking at associated  co-t-structures.
 
\begin{definition}
A classic (respectively, large) presilting object $X$ in a triangulated category $\sT$ is said to \textit{admit a complement} if there is an object $V$ in $\sT$ such that $X\oplus V$ is a classic (respectively, large) silting object. 
\end{definition}

\subsection{Intermediate (pre)silting objects}

We recall from \cite{AI12} the following relation on classic presilting subcategories.
When the subcategories are classic silting, this relation defines a partial order \cite[Theorem 2.11]{AI12}. 

\begin{definition} \label{partial-order}
For two classic presilting subcategories $\sX$ and $\sY$ in a triangulated category $\sT$, we set
\[
\sX \geq \sY  \iff \Hom_\sT (X, Y[>0])=0 \text{ for all } X \in \sX \text{ and } Y \in \sY.
\]
In our notation, we have that $\sX \geq \sY \iff  \sY \subseteq \sV_\sX$.
\end{definition}

It follows that if $\sX$ is a classic presilting subcategory in $\sT$, then $\sX[-1] \geq \sX \geq \sX[1]$. We need the following minor generalisation of \cite[Lemma 3.6]{AMY19}.

\begin{lemma}\label{interm}
Let $\sT$ be a triangulated category. Given a classic presilting subcategory $\sX$ and    a  classic silting subcategory $\sY$ of $\sT$, we have that
\begin{enumerate}[label=(\alph*)]
\item  $\sX \geq \sY$ if and only if $\sV_\sY \subseteq \sV_\sX;$ 
\item  for any $n>0$, $\sY\ge\sX\ge \sY[n]$ if and only if $\sX\subseteq\sY*\sY[1]*\ldots*\sY[n]$.
\end{enumerate}
The same statements hold when
$\sT$ is a compactly generated triangulated category, $\sX=\Add{X}$ for a large presilting object $X$, and 
$\sY=\Add{Y}$ for a large silting object $Y$ in $\sT$ with $X$ lying in $\thick{}{\sY}$.
\end{lemma}

If condition (b) holds for $\sX$ and $\sY$ as in the lemma then we say that $\sX$ is \emph{intermediate with respect to $\sY$}.

\begin{proof}
(a) $\sV_\sX$ is closed under suspensions, extensions and direct summands, thus $\sY \subseteq \sV_\sX$ implies $\susp(\sY) \subseteq \sV_\sX$, and the claim follows because $\sV_\sY =\susp(\sY)$; see Theorem \ref{MSSS}.

(b)  The if part is clear. For the converse, observe that  
\[
\sT=\thick{}{\sY} = \bigcup_{k\geq 0} \sY[-k] * \sY[-k+1] * \cdots * \sY[k],
\]  
where the last equality holds due to \cite[Lemma 2.6]{IYa18} (see also \cite[Lemma 3.5(c)]{AMY19}). Thus, there is $k\geq 0$ for which $X$ lies in $\sY[-k]*\sY[-k+1]*\ldots*\sY[k]$. Assuming that there is $n>0$ such that $\sY\ge\sX\ge \sY[n]$, we obtain the statement using \cite[Lemma 3.5(b)]{AMY19}.

Finally, in the context of the last assertion, we can use the same arguments to prove (a), taking into account that  $\sV_X$ is closed under coproducts, and  that $\sV_\sY=\Susp{Y}$ by Theorem \ref{large correspondence}. For (b), we observe that $\sX=\Add{X}$ is a classic presilting subcategory in $\thick{}{\sY}$ and $\sY=\Add{Y}$ is a classic silting subcategory in $\thick{}{\sY}$ by Remark~\ref{classic silting in thick}.
\end{proof}

Given a silting object $M$, a useful recent result of Breaz allows us to identify large silting objects $N$ which are intermediate with respect to $M$ in a slightly less cumbersome manner.
We will use this result in the Section 5.

\begin{theorem}[{\cite[Theorem 3.4]{Breaz}}] \label{Breaz}
Let $\sT$ be a triangulated category with coproducts and $M$ a large silting object in $\sT$. The following are equivalent for an object $N$ in $\sT$.
\begin{enumerate}
\item The object $N$ is a large silting object such that $\Add{M} \geq \Add{N} \geq \Add{M[n]}$;
\item The object $N$ satisfies:
\begin{enumerate}[label=(\roman*)]
\item $N$ lies in $\sV_M$ and there is $n>0$ such that $\sV_M[n]\subseteq \sV_N$;
\item $N$ is a weak generator; and,
\item $\Add{N}$ lies in $\sV_N$.
\end{enumerate}
\end{enumerate}
\end{theorem}

\subsection{Complements for large presilting}

We are now ready to state the first result concerning the existence of complements in the context of algebraic, compactly generated triangulated categories. 
Recall from \cite{Kel94} that a triangulated category is \emph{algebraic} if it is equivalent to the stable category of a Frobenius exact category.
We first need the following observation from \cite{SS11}.

\begin{lemma} \label{saorin-stovicek}
Let $\sT$ be an algebraic, compactly generated triangulated category. Suppose $\cS$ is a set of objects such that $\cS[-1] \subset \cS$ (resp.~$\cS[1] \subset \cS$). Then, $({}\orth(\cS\orth), \cS\orth)$ is a co-t-structure (resp.~t-structure) in $\sT$.
\end{lemma}

\begin{proof}
By \cite[\S 4.3, Theorem]{Kel94}, $\sT$ is equivalent to the derived category of a small differential graded category. From \cite[Remark 2.15]{SP16} it follows that there $\sT$ can be seen as the stable category of an efficient Frobenius exact category in the sense of \cite[Definition 2.6]{SS11}. Finally, by \cite[Proposition 3.3 and Corollary  3.5]{SS11}, any set $\cS$ of objects such that $\cS[-1] \subset \cS$ (resp.~$\cS[1] \subset \cS$) gives rise to a co-t-structure (resp.~t-structure) $({}^\perp(\mathcal{S}\orth),\mathcal{S}\orth)$.
\end{proof}

\begin{theorem}\label{large completion}
Let $\sT$ be an algebraic compactly generated triangulated category, and suppose that $X$ is a large presilting object in $\sT$. The following statements are equivalent.
\begin{enumerate}
\item $X$ admits a complement $V$ such that $X\oplus V$ is a large silting object in $\sT$.
\item There is a large silting object $M$ in $\sT$ such that $X$ lies in $\thick{}{\Add{M}}$.
\item There are a large silting object $M$ in $\sT$ and an integer $n>0$ such that 
\[
\Add{M}\geq \Add{X}\geq \Add{M[n]}.
\]
\end{enumerate}
Moreover, for any $M$ satisfying the equivalent conditions (2) and (3) above, there exists a complement $V$ such that $\thick{}{\Add{M}}=\thick{}{\Add{X\oplus V}}.$
\end{theorem}

\begin{proof}
(1) $\Rightarrow$ (2): If $X$ admits a complement, $V$ say, then $M=X\oplus V$ is a large silting object satisfying (2).

(2) $\Rightarrow$ (3): If $X$ lies in $\thick{}{\Add{M}}$, it follows from \cite[Lemma 3.5(c)]{AMY19} and Remark~\ref{classic silting in thick} that $X$ lies in $\Add{M}[-k]*\Add{M}[-k+1]*\cdots *\Add{M}[k]$ for some $k>0$. Choose $M':=M[-k]$ and it then follows from Lemma \ref{interm}(b) that 
$\Add{M'} \geq \Add{X} \geq \Add{M'[2k]}$.

(3) $\Rightarrow$ (1):
First note that, as $\sT$ is compactly generated, by Theorem \ref{large correspondence} there are co-t-structures $(\sU_X,\sV_X)$ and $(\sU_M,\sV_M)$ with cohearts $\Add{X}$ and $\Add{M}$, respectively. Aplying Lemma~\ref{saorin-stovicek} on the set  $\cS = \{ M[k],X[k] \mid k<0\}$, we obtain a co-t-structure
\[
(\sU, \sV := \sV_M \cap \sV_X).
\]
We write $\sC := \sU[1] \cap \sV$ for its coheart. 

Consider now a decomposition of $M$ with respect to this co-t-structure 
\begin{equation} \label{complement}
\begin{tikzcd}
U \ar[r, "\Phi"] & M \ar[r, "\Psi"] & V \ar[r] & U[1]
\end{tikzcd}
\end{equation}
with $U$ in $\sU$ and $V$ in $\sV$. We claim that $V$ is a complement for $X$. 
We proceed in a sequence of steps.

{\bf Step 1:} {\it We have $\sC = \Add{X\oplus V}$.}

We first show that  $\Add{X\oplus V} \subseteq \sC$. The object $X$ lies in both $\sV_X$  and $\sV_M$ since it is large presilting and since $\Add{M}\geq \Add{X}$ by assumption, respectively. Moreover, since $\sV \subseteq \sV_X$, it follows that $\Hom_{\sT}(X,\sV[1])=0$, whence $X$ lies in $\sU[1] = {}\orth \sV[1]$ and we have that $X$ lies in $\sC$.
Similarly, as $M$ belongs to $\sU[1] = {}\orth \sV[1]$, we observe from the triangle (\ref{complement}) that $V$ lies in $\sU[1]$ and, thus, in $\sC$.
As $\sC$ is closed under coproducts and direct summands, we see that $\Add{X\oplus V} \subseteq \sC$. 

For the reverse inclusion, consider  an object $C$ of $\sC$ together with an  $\Add{X\oplus V}$-precover $\phi \colon K\to C$, which exists since $\sT$ has set-indexed coproducts. We have a triangle 
\[
\begin{tikzcd}
K \ar[r,"\phi"] & C \ar[r,"\psi"] & L \ar[r, "\theta"] & K[1]
\end{tikzcd}
\] 
and we claim that $L$ lies in $\sV[1]$. It is clear that $L$ lies in $\sV$ since $\sV$ is a suspended subcategory of $\sT$. Therefore, it remains only to show that $\Hom_{\sT}(X,L)=0=\Hom_{\sT}(M,L)$. For any map $f \colon X\longrightarrow L$, we have that $\theta f = 0$ since $K$ lies in $\sV$. Therefore, there is a map $\overline{f}\colon X\longrightarrow C$ such that $f=\psi\overline{f}$. But $\phi$ is an $\Add{X\oplus V}$-precover and therefore $\overline{f}$ must factor through $\phi$, thus showing that $f=0$ and, hence, $\Hom_{\sT}(X,L)=0$.
On the other hand, for any map $g\colon M\longrightarrow L$, we also have that $\theta g=0$. Thus, there is a map $\overline{g} \colon M\longrightarrow C$ such that $g=\psi\overline{g}$. Since $C$ lies in $\sV$, the map $\overline{g}$ must factor through the $\sV$-preenvelope $\Psi$ from (\ref{complement}), i.e. there is $\hat{g} \colon V \longrightarrow C$ such that $g = \psi \hat{g} \Psi$. Again, because $\phi$ is an $\Add{X\oplus V}$-precover, $\hat{g}$ must factor through $\phi$, showing that $g=0$. Hence $L$ lies in $\sV[1]$ as claimed, and thus $\psi=0$ and $\phi$ is a split epimorphism. This  proves that $C$ lies in $\Add{X \oplus V}$.

{\bf Step 2:} {\it The object $X \oplus V$ is a large presilting}.

As $X\oplus V$ lies in the coheart of a co-t-structure, it satisfies $\Hom_\sT(X\oplus V,X\oplus V[>0])=0$. 
Since $\sV$ is closed under coproducts, it suffices to show $\sV_{X\oplus V}  =\sV$.
As $X\oplus V$ lies in $\sC \subseteq \sU[1]$, we have that $\sV \subseteq \sV_{X\oplus V}$. To see the reverse inclusion, we recall that by assumption 
and Lemma~\ref{interm}(a) there is $n>0$ such that  $\sV_{M[n]}\subseteq \sV_X$, and since $\sV_M[n]=\sV_{M[n]}$ and $\sV=\sV_X\cap\sV_M$, we find that 
\[
\sV_M[n] \subseteq \sV \subseteq \sV_M.
\]
By Corollary~\ref{intermed}
we have that $M$ lies in $\sC[-n] * \cdots * \sC[-1] * \sC$. We conclude then that $\sV_\sC \subseteq \sV_M$, and since $\sC = \Add{X\oplus V}$ by Step 1, we get the desired inclusion. 

{\bf Step 3:} {\it The object $X\oplus V$ is a weak generator.}

In order to conclude that $X\oplus V$ is a large silting object, by Remark \ref{large silting = large presilting + weak generation} it suffices to show that $X\oplus V$ is a weak generator of $\sT$. We have seen in Step 2 that $M$ lies in $\thick{}{\Add{X\oplus V}}$. Hence, if $Y$ lies in $(X\oplus V)[\bZ]\orth$,  it also lies in $M[\bZ]\orth$. As $M$ is a weak generator for $\sT$, it follows that $Y=0$, as required. 

We conclude that $X \oplus V$ is a large silting object and that $V$ is a complement for $X$. This completes the proof of the equivalence of conditions (1), (2) and (3).

Finally, as shown above, for any large silting $M$ satisfying the equivalent conditions (2) and (3), we can find a complement $V$ of $X$ such that 
$\sV_M[n] \subseteq \sV_{X\oplus V} \subseteq \sV_M,$ and it  follows from Corollary~\ref{intermed} and Theorem \ref{large correspondence} that 
$\thick{}{\Add{M}}= \thick{}{\Add{X\oplus V}}$. 
\end{proof}

\begin{corollary} \label{cor2}
Let $R$ be a coherent ring and let $X$ be a large presilting object in $\sD(R)$. If $X$ is a bounded complex of projective $R$-modules, then it admits a complement which is also a bounded complex of projective $R$-modules.
\end{corollary}

\begin{proof}
This follows directly from Theorem \ref{large completion} since $R$ is a large silting object in $\sD(R)$ and $X$ is a bounded complex of projective $R$-modules if and only if $X$ lies in $\thick{}{\Add{R}}$. 
\end{proof}

\begin{remark} \label{small admits large complement}
As a consequence of Corollary~\ref{cor2}, every classic presilting object in $\per{R}$ admits a complement if we extend the ambient category to $\sD(R)$. In other words, classic presilting objects in $\per{R}$ admit complements if we regard them as large presilting objects in $\sD(R)$, cf.~Remark~\ref{comparison}. Moreover, these complements can always be found in $\mathsf{K}^b(\mathsf{Proj}(R))$.
\end{remark}

\subsection{Complements to classic presilting}

To establish a criterion for the existence of complements in the classic setting, we will imitate the strategy of Theorem \ref{large completion}. For this purpose, we need to associate co-t-structures to classic presilting objects. Fortunately, this happens frequently, as shown in the following proposition, which is essentially a reformulation of \cite[Proposition 3.2]{IYa18}.

\begin{proposition}\label{classic to special}
Let $\sT$ be a triangulated category containing a classic silting subcategory. Let $X$ be an object of $\sT$. The following statements hold.
\begin{enumerate}
\item If $\sX=\add{X}$ is precovering in $\sT$, then $\sX$ is a classic presilting subcategory if and only if  $(\sU_\sX,\sV_\sX)$ is a co-t-structure in $\sT$ with coheart $\sX$.
\item If $\sT$ admits set-indexed self-coproducts, then $\sX=\Add{X}$ is a classic presilting subcategory if and only if $(\sU_\sX,\sV_\sX)$ is a co-t-structure in $\sT$ with coheart $\sX$.
\end{enumerate}
\end{proposition}

\begin{proof}
Let  $\sX$ be a classic presilting subcategory in $\sT$. Observe first that assumption (P2) from \cite[p.~7870]{IYa18} holds by Lemma~\ref{only finite} since we assume that $\sX$ is additively generated by a single object. Moreover, by the proof of \cite[Proposition 3.2]{IYa18}, the following conditions are equivalent:
\begin{itemize}
\item the subcategory $\sX$ is precovering in $\sV_\sX$, and assumption (P2) holds,
\item $(\sU_\sX,\sV_\sX) $ is a co-t-structure in $\sT$ with  $\sU_\sX=\cosusp{\sX[-1]}$,
\end{itemize}
and under these conditions $(\sU_\sX,\sV_\sX)$  has coheart $\sX$.  This yields  statement (1).

The proof of statement (2) is analogous once one observes that the existence of self-coproducts in $\sT$ guarantees the existence of $\Add{X}$-precovers. Indeed, for any object $T$ of $\sT$ an $\Add{X}$-precover is given by taking the universal map $\phi\colon X^{(\Hom_{\sT}(X,T))} \too T$. 
\end{proof}

Before stating our criterion for the existence of complements for classic presilting objects we need the following straightforward lemma.

\begin{lemma} \label{summand}
Suppose $M$ is a classic silting object and $X$ is a classic presilting object of $\sT$. Then $X$  lies in $\add{M}$  if and only if $\add{M} \geq \add{X}$ and $\add{X} \geq \add{M}$.
\end{lemma}

\begin{proof}
If $X$ is an object of $\add{M}$ then the relations $\add{M} \geq \add{X}$ and $\add{X} \geq \add{M}$ are clear.
Conversely, suppose $\add{M} \geq \add{X}$ and $\add{X} \geq \add{M}$. Since $\add{M} \geq \add{X}$, we have that $X$ lies in $\sV_M $, and so we can decompose $X$ as 
\[
M_1 \to X \to V_1[1] \to M_1[1]
\] 
with $M_1$ in $\add{M}$ and $V_1$ in $\sV_M$. As $\add{X} \geq \add{M}$, the morphism $X \to V_1[1]$ must be zero. Hence, the triangle splits and $X$ is a direct summand of $M_1$ and thus an object of $\add{M}$.
\end{proof}

\begin{theorem} \label{small completion}
Let $\sT$ be a triangulated category such that $\add{T}$ is precovering for any object $T$ in $\sT$. A classic presilting object $X$ in $\sT$ admits a complement if and only if there is a classic silting object $M$ of $\sT$ such that $\add{M} \geq \add{X}$ and for which the pair
\[
(\sU := {}\orth\sV, \sV := \sV_X\cap\sV_M),
\]
is a co-t-structure. In this case we have that $\add{X}\geq \add{M[n]}$ for some $n \geq 0$. 
\end{theorem}

Note that the condition that $\add{T}$ is precovering for any object $T$ is automatically satisfied whenever $\sT$ is a $\kk$-linear, Hom-finite triangulated category over a field $\kk$.

\begin{proof}
If $X$ admits a complement, $V$ say, then it is clear that $M = X \oplus V$ is a classic silting object satisfying the required conditions. 

For the converse implication, suppose that $M$ is a classic silting object $M$ for which $\add{M} \geq \add{X}$ and such that $(\sU,\sV := \sV_X\cap\sV_M)$ is a co-t-structure.
If $\add{X} \geq \add{M}$ then $X \in \add{M}$ by Lemma~\ref{summand} and $M$ is, itself, a complement. 
Therefore, we assume that $\add{X} \ngeq \add{M}$, in which case $M$ does not lie in  $\sV$.
As in the proof of Theorem \ref{large completion}, we find a complement by truncating $M$ with respect to $(\sU,\sV)$ 
\begin{equation} \label{small complement}
\begin{tikzcd}
U \ar[r, "\Phi"] & M \ar[r, "\Psi"] & V \ar[r] & U[1].
\end{tikzcd}
\end{equation}
Arguing as in Step 1 of the proof of Theorem \ref{large completion}, thanks to the assumption that the subcategory $\add{X\oplus V}$ is precovering, we conclude that $\add{X\oplus V}$ is the coheart of $(\sU,\sV)$. In particular, $X\oplus V$ is a classic presilting object.

To see that $X \oplus V$ is a classic silting object, it is enough to check that $\thick{}{M}=\thick{}{X\oplus V}$. 
By Lemma \ref{only finite} there is  $n>0$ such that $\Hom_{\sT}(X,M[{>n}])=0$, and 
we obtain $\add{M}\geq \add{X}\geq \add{M[n]}$. We can now proceed as in Step 2 of  the proof of Theorem~\ref{large completion} to conclude that $\sV_M[n] \subseteq \sV \subseteq \sV_M,$ and  Corollary~\ref{intermed}  yields $\thick{}{M}=\thick{}{X\oplus V}$, as desired.
\end{proof}

The following corollary applies to categories such as $\Kb(\Proj{R})$ or $\Db(\Mod{R})$ which have the property that every object in them admits any set-indexed self-coproduct.

\begin{corollary} \label{cor2small}
Let $\sT$ be a triangulated category  admitting set-indexed self-coproducts. If $X$ is an object such that $\Add{X}$ is a classic presilting subcategory, then there is a classic silting subcategory $\sN$ containing $\Add{X}$ if and only if there is a classic silting subcategory $\sM$ of $\sT$ such that $\sM \geq \Add{X}$ and there is a co-t-structure of the form 
$
(\sU,\sV := \sV_X \cap \sV_{\sM}).
$
\end{corollary}

\begin{proof}
The existence of arbitrary self-coproducts in $\sT$ implies that $\Add{Z}$ is precovering for any object $Z$ in $\sT$. One can now apply the argument of the proof of Theorem~\ref{small completion} 
noting that as $\sM$ is a classic silting subcategory, the classic silting subcategory $\sN$ is constructed as $\Add{X \oplus V_M \mid M \in \sM}$, where $V_M$ is constructed via the truncation triangles \eqref{small complement} as a $\sV$-preenvelope of $M\in \sM$.
Furthermore, the co-t-structure associated to $\sM$ by Theorem~\ref{MSSS} is bounded, so that there exists $n>0$ with $\Hom_{\sT}(X,\sM[>n])=0$ despite $\sM$ being a silting subcategory rather than a silting object.
\end{proof}

In light of the (proofs of) Theorems~\ref{large completion} and \ref{small completion} we make the following definition.

\begin{definition}
Suppose $X$ is a large (resp.~classic) presilting object and $M$ is a large (resp.~classic) silting object. We call the large (resp.~classic) silting object $V \oplus X$ constructed in the proof of Theorem~\ref{large completion} (resp.~Theorem~\ref{small completion}) the \emph{completion of $X$ with respect to $M$}.
\end{definition}

%============================================================================
% SECTION
\section{Applications} \label{sec:applications}
%============================================================================

In this section we provide two immediate applications of our results  in Section~\ref{sec:complements}, namely, new proofs of the existence of complements for classic presilting objects in  hereditary categories and  of the classic Bongartz completion lemma for `two-term' presilting subcategories.

We recall the following well-known characterisation of split t-structures in terms of their hearts, cf. \cite[Lemma 2.1 \& Theorem 2.3]{CR18} and \cite[Proposition 1]{Kel05}. We include the argument for the convenience of the reader.

\begin{proposition} \label{prop:hereditary}
Let $\sT$ be a triangulated category and let $(\sV,\sW)$ be a bounded t-structure in $\sT$ with heart $\sA$ and associated cohomological functor $H \colon \sT \to \sA$. Write $H^i(X)=H(X[i])$ for any $X$ in $\sT$. The following conditions are equivalent.
\begin{enumerate}[label=(\arabic*)]
\item For all objects $A_1$ and $A_2$ of $\sA$, we have $\Hom_{\sT}(A_1,A_2[2]) = 0$.
\item For each object $T$ of $\sT$, we have $T \cong \bigoplus_{i \in \bZ} H^{i}(T)[-i]$.
\item The t-structure $(\sV,\sW)$ is split.
\end{enumerate}
\end{proposition}

\begin{proof}
$(1) \implies (2)$. Observe that $\sA[1] * \sA = \add{\sA[1],\sA}$. The inclusion $\sA[1] * \sA \supseteq \add{\sA[1],\sA}$ is clear, while the inclusion $\sA[1] * \sA \subseteq \add{\sA[1],\sA}$ follows immediately from the condition $\Hom_{\sT}(\sA,\sA[2]) = 0$. Finally, the characterisation of the boundedness of $(\sV,\sW)$ via
\[
\sT = \bigcup_{i \geq j} \sA[i] * \sA[i-1] * \cdots * \sA[j],
\]
and induction shows that 
\[
\sT = \bigcup_{i \geq j} \add{\sA[i], \sA[i-1], \ldots, \sA[j]}, 
\]
from which we see that each object of $\sT$ decomposes into a direct sum of its cohomology with respect to $(\sV,\sW)$.

$(2) \implies (3)$. Write $T \cong \bigoplus_{i \in \bZ} H^{i}(T)[-i]$, the split triangle $\bigoplus_{i < 0} H^i(T)[-i] \to T \to \bigoplus_{i \geq 0} H^i(T)[-i] \to \big(\bigoplus_{i \le 0} H^i(T)[-i] \big)$ gives the truncation triangle for $T$.

$(3) \implies (1)$. Take objects $A_1$, $A_2$ of $\sA$ and extend a morphism $A_1 \to A_2[2]$ to the triangle $A_2[1] \to C \to A_1 \to A_2[2]$. As $A_2[1] \in \sA[1] \subset \sV$ and $A_1 \in \sA \subset \sW$, this triangle is the truncation triangle of $C$ with respect to the split t-structure $(\sV,\sW)$, in which case the third map is zero.
\end{proof}

We will be considering abelian categories $\sA$ containing a projective object $P$ such that every object of $\sA$ is a quotient of an object in $\add{P}$. 
In such categories, it is well known that if $\sA$ has finite global dimension, then the bounded derived category $\Db(\sA)$ is equivalent to the bounded homotopy category $\mathsf{K}^b(\mathsf{proj}(\sA))$ of the additive category $\mathsf{proj}(\sA)=\add{P}$ of projective objects in $\sA$. This means, in particular, that $P$ is a classic silting object in $\Db(\sA) = \Kb(\mathsf{proj}(\sA))$.

The following result intersects non-trivially \cite[Theorem 1.2]{DaFu22}, which was proved using other methods.
Recall that $\sA$ is \emph{hereditary} if $\Ext^2_{\sA}(-,-) = 0$.

\begin{proposition}[Hereditary silting completion] \label{hereditary completion}
Suppose $\sA$ is a hereditary abelian category with a projective object $P$ such that every object in $\sA$ is a quotient of an object in $\add{P}$. Suppose further that $\add{T}$ is precovering for every object $T$ in $\Db(\sA)$.
If $X$ is a classic presilting object in $\Db(\sA)$, then $X$ admits a complement to a classic silting object in $\sD^b(\sA)$.
\end{proposition}

\begin{proof}
Under the hypotheses of the proposition, it is clear that $P$ is a classic silting object in $\Db(\sA)$, and it is well known that the associated torsion pairs $(\sV_P,\sW_P)$ and $(\sU_P,\sV_P)$ in $\Db(\sA)=\thick{}{P}=\mathsf{K}^b(\add{P})$ are, respectively, the standard t-structure and its left adjacent co-t-structure. Note that the corresponding truncation triangles are given by the so-called \textit{smart} and \textit{stupid} truncations. 
Consider now the co-t-structure $(\sU_X,\sV_X)$ associated to the classic presilting object $X$ (see Proposition \ref{classic to special}(1)).
We will show that $(\add{\sU_P * \sU_X},\sV_P \cap \sV_X)$ is a co-t-structure in $\Db(\sA)$. Closure under shifts and Hom-orthogonality are clear. Closure under summands is given by definition for $\add{\sU_P * \sU_X}$ and is clear for $\sV_P \cap \sV_X$ because it is defined by an orthogonality condition. It remains to obtain the decomposition triangle. Let $D$ be an object of $\Db(\sA)$, consider the truncation triangle with respect to $(\sU_P, \sV_P)$,
\[
U_P \too D \rightlabel{f}  V_P \too U_P[1].
\]
Now we truncate $V_P$ with respect to $(\sU_X,\sV_X)$:
\[
U_X \too V_P \rightlabel{g}  V_X \too U_X[1].
\]
Finally, we truncate $V_X$ with respect to the standard t-structure $(\sV_P,\sW_P)$: 
\[
\widetilde{V}_P \rightlabel{h} V_X \too W_P \rightlabel{0} \widetilde{V}_P[1],
\] 
where $\widetilde{V}_P$ lies in $\sV_P$, $W_P$ lies in $\sW_P$ and the third morphism is $0$ by Proposition~\ref{prop:hereditary} since $\sA$ is hereditary and $\Hom_{\Db(\sA)}(A_1, A_2[2]) = \Hom_{\sA}(A_1,A_2[2])$ for objects $A_1$ and $A_2$ of $\sA$. In particular, $\widetilde{V}_P$ is a direct summand of $V_X$ and therefore lies in $\sV_X$ and hence lies in $\sV_P \cap \sV_X$. 
As $(\sV_P, \sW_P)$ is a t-structure, $h \colon \widetilde{V}_P \to V_X$ is a $\sV_P$-precover. Thus, there exists $\tilde{g} \colon V_P \to \widetilde{V}_P$ such that $g = h\tilde{g}$.
Applying the octahedral axiom to the composition $g = h\tilde{g}$ gives the following commutative diagram.
\[
\begin{tikzcd}[sep=small]
                                                      & W_P[-1] \ar[d, swap, "0"] \ar[r,equals] & W_P[-1] \ar[d, "0"]  & \\
V_P \ar[r, "\tilde{g}"] \ar[d,equals] & \widetilde{V}_P \ar[r] \ar[d,swap, "h"]   & C[1] \ar[r] \ar[d]       & V_P[1]  \ar[d,equals] \\
V_P \ar[r, swap, "g"]                      & V_X \ar[r] \ar[d]                                    & U_X[1] \ar[r]  \ar[d]  & V_P[1] \\
                                                      & W_P \ar[r, equals]                               & W_P                        & 
\end{tikzcd}
\]
The split triangle forming the third column shows that $C$ is a direct summand of $U_X$ and thus lies in $\sU_X$. 
Now applying the octahedral axiom to the composition $\tilde{g}f$ gives the following commutative diagram.
\[
\begin{tikzcd}[sep=small]
                                          & C \ar[d] \ar[r,equals]                      & C \ar[d]                    & \\
D \ar[r, "f"] \ar[d,equals]     & V_P \ar[r] \ar[d,swap, "\tilde{g}"]   & U_P[1] \ar[r] \ar[d]   & D[1]  \ar[d,equals] \\
D \ar[r, swap, "\tilde{g}f"]   & \widetilde{V}_P \ar[r] \ar[d]            & U[1] \ar[r]  \ar[d]       &  D[1] \\
                                          & C[1] \ar[r, equals]                           & C[1]                         & 
\end{tikzcd}
\]
Observe that $U \in \sU_P * \sU_X$, and hence 
\[
U \too D \rightlabel{\tilde{g}f} \widetilde{V}_P \too U[1]
\]
is a truncation triangle showing that $(\add{\sU_P * \sU_X},\sV_P \cap \sV_X)$ is a co-t-structure in $\Db(\sA)$ and the result follows by Theorem~\ref{small completion}.
\end{proof}

\begin{remark}
Note that if $\sA$ is a cocomplete hereditary abelian category  with a projective generator $P$, then $P$ is a silting object in $\sD(\sA)$ and $\Db(\sA)\cong \Kb(\mathsf{Proj}(\sA))$ (see, for example, \cite[\S4]{PV18}). Thus, if $\sD(\sA)$ is a compactly generated triangulated category, then a large presilting complex $X$ in $\Db(\sA)$ admits a complement to a large silting object in $\sD(\sA)$ following Theorem \ref{large completion} (just as argued in Corollary \ref{cor2}). Note, furthermore, that the complement found using Theorem \ref{large completion} is an object in $\Db(\sA)$.
\end{remark}

Next, we recover the classic Bongartz completion lemma for two-term classic presilting objects, see \cite[\S 5]{DeFe15}, \cite[Proposition 6.1]{Wei13}, \cite[Proposition 3.14]{BY}.

\begin{proposition}[Bongartz completion] \label{bongartz}
Let $\sT$ be a triangulated category such that $\add{T}$ is precovering  for any object $T$ in $\sT$ (for example,  a $\kk$-linear, Hom-finite triangulated category over a field $\kk$). Let $M$ be a classic silting object and $X$  a classic presilting object in $\sT$ which is two-term with respect to $M$, i.e.
\[
\add{M}\ge\add{X}\ge\add{M[1]}.
\]
Then $\sV_M \cap \sV_X$ is the coaisle of a co-t-structure. In particular, there is $V$ in $\sT$ such that $X\oplus V$ is a classic silting object with $\sV_{X\oplus V} = \sV_M \cap \sV_X$.
\end{proposition}

\begin{proof}
We show that $(\add{\sU_M * \sU_X}, \sV_M \cap \sV_X)$ is a co-t-structure in $\sT$. Closure under shifts, closure under summands and Hom-orthogonality are clear. It remains to obtain a decomposition triangle for each object $T$ in $\sT$. 
We truncate first with respect to $(\sU_M, \sV_M)$ and then with respect to $(\sU_X, \sV_X)$, which is a co-t-structure with coheart $\add{X}$ by Proposition~\ref{classic to special}. We obtain triangles
\[
U_M \too T \rightlabel{f} V_M \too  U_M[1]
\quad \text{and} \quad
U_X \too V_M \rightlabel{g} V_X \too U_X[1],
\]
with $U_M$ in $\sU_M$, $V_M$ in $\sV_M$, $U_X$ in $\sU_X$ and $V_X$ in $\sV_X$. 
By assumption and Lemma~\ref{interm}(a), we see that $V_M[1]$ lies in $\sV_X$, and so does $U_X[1]$. 
Thus $U_X[1]$ lies in $\sC_X = \add{X} \subseteq \sV_M$.
Hence, $V_X$ lies in $\sV_M \cap \sV_X$. Using the octahedral axiom, we get a triangle
\[
U \too T \rightlabel{gf} V_X \too U[1],
\]
with $U$ lying in $\sU_M * \sU_X$, and we conclude that $(\add{\sU_M * \sU_X}, \sV_M \cap \sV_X)$ is a co-t-structure. The existence of a complement now follows from Theorem~\ref{small completion}.
\end{proof}

\begin{remark}
In each of Propositions~\ref{hereditary completion} and \ref{bongartz}, we obtain a truncation triangle for the co-t-structure $(\add{\sU_M * \sU_X}, \sV_M \cap \sV_X)$ by first truncating an object of $\sT$ with respect to $(\sU_M, \sV_M)$ and then truncating the resulting object of $\sV_M$ with respect to $(\sU_X, \sV_X)$. One then observes, using two different arguments, that the object of $\sV_X$ resulting from the second truncation is also an object of $\sV_M$.
This is an example in which the naive truncation algorithm of \cite[\S 2]{BPP13} terminates after two steps; see also \cite{Bondal13}.
It would be interesting to find conditions under which the naive or refined truncation algorithm (\cite[\S 3]{BPP13}) terminates after finitely many steps, e.g. \cite[Theorem 6.1]{BPP13}.
\end{remark}

The (complete) silting object with respect to which the complement is taken matters.
Attempting to complete a given presilting object with respect to different (complete) silting objects may yield different completions, or even a completion in one case and no completion in the other. The examples below exhibit both of these kinds of behaviour.

\begin{example}
Let $Q$ be the $A_3$ quiver below and let $\kk Q$ be its path algebra.
\[
\begin{tikzcd}[sep=small]
1 \ar[r]  & 2 \ar[r]  & 3
\end{tikzcd}
\]
Let $X = S_1$, which is a classic presilting object in $\Db(\kk Q)$. We consider completions of $X$ with respect to two different classic silting objects in $\Db(\kk Q)$. Consider the classic silting object $M = P_1 \oplus P_2 \oplus P_3$. Then applying Theorem~\ref{small completion}, we obtain the completion $V \oplus X = P_1 \oplus P_3 \oplus S_1$ from the decomposition triangle 
\[
S_1[-1] \to P_3 \oplus P_2 \oplus P_3 \to P_3 \oplus P_1 \oplus P_1 \to S_1.
\]
Now consider the classic silting object $N = P_2 \oplus P_1 \oplus S_2$. In this case, applying Theorem~\ref{small completion} gives us the completion $V' \oplus X = P_1 \oplus I_2 \oplus S_1$ from the decomposition triangle
\[
S_1[-1] \oplus S_1[-1] \to P_2 \oplus P_1 \oplus S_2 \to P_1 \oplus P_1 \oplus I_2 \to S_1 \oplus S_1.
\]
Note that the complement need not be basic; in each case, we have given the basic version.
\end{example}

The following example is based on the method in \cite{BPP13}.

\begin{example} \label{A2}
Let $Q$ be the $\widetilde{A}_2$ quiver below and let $\kk Q$ be its path algebra.
\[
\begin{tikzcd}[sep=small]
                       & 2 \ar[dr] & \\
1 \ar[rr] \ar[ur] &              & 3
\end{tikzcd}
\]
The object $P_1 \oplus P_3 \oplus \tau S_2$ is a tilting object in $\Db(\kk Q)$, from which we deduce that $M \coloneqq P_1 \oplus P_3 \oplus \tau S_2[1]$ is a classic silting object in $\Db(\kk Q)$.
Let $X \coloneqq S_2 [2]$. Since $S_2$ is rigid, $X$ is a classic presilting object in $\Db(\kk Q)$ which lies in $\susp M = \sV_M$ and $M[2] \in \sV_X$. By Proposition~\ref{classic to special},  $(\cosusp X[-1], \sV_X)$ is a co-t-structure in $\Db(\kk Q)$. Consider
\[
\sV \coloneqq \sV_X \cap \sV_M  = \sV_{X \oplus M}.
\]
The suspended subcategory $\sV$ is not preenveloping in $\Db(\kk Q)$ and therefore it is not the co-aisle of a co-t-structure, see Figure~\ref{A2-example} on page~\pageref{A2-example} for an illustration. 
Hence, $X$ cannot be completed \emph{with respect to} $M$. However, $X$ can be completed with respect to $N \coloneqq (P_1 \oplus P_2 \oplus P_3)[1]$ because $X$ is two-term with respect to $N$ (see Proposition \ref{bongartz}).
\end{example}

%============================================================================
% SECTION
\section{Silting-discrete finite-dimensional algebras} \label{sec:silting-discrete}
%============================================================================

In this section, $\Lambda$ will be a finite-dimensional algebra over a field $\kk$.

\subsection{Classic silting objects versus large silting objects}
We have seen in Remark~\ref{comparison} that an object 
 $M$ in $\per{\Lambda}$ is a classic silting object in $\per{\Lambda}$ if and only if it is a large silting object in $\sD(\Lambda)$. Consider the following  pairs
 \[
(\sV_M, \sW_M) 
=\big( (M[<0])\orth, (M[\geq 0])\orth \big)
\text{ and }
(\sv_M, \sw_M)  \coloneqq \big( \sV_M \cap \Db(\Lambda), \sW_M \cap \Db(\Lambda) \big),
\]
where the orthogonals are taken inside $\sD(\Lambda)$.

\begin{proposition}[{\cite[Theorem 1.3]{HKM02}, \cite[Corollary 2]{NSZ19}, \cite[Proposition 4.3]{PV18}, \cite[Proposition 5.4]{KY14}}]\label{Koenig-Yang}
Let $M \in \per{\Lambda}$ be a classic silting object. 
\begin{enumerate}
\item  The pair of subcategories $(\sV_M, \sW_M)$ is a t-structure in $\sD(\Lambda)$;
\item  \label{projective-generator} The cohomological functor $H^0_M \colon \sD(\Lambda)\to \sH_M$ 
associated to the t-structure $(\sV_M,\sW_M)$ induces an equivalence 
\[
H^0_M|_{\Add{M}} \colon \Add{M} \to \Proj{\sH_M}.
\]
In particular, $H^0(M)$ is a small projective generator of $\sH_M$, and the heart $\sH_M$ is equivalent to $\Mod{\End(H^0(M))}$. 
\item  The pair $(\sv_M,\sw_M)$ is a t-structure in $\Db(\Lambda)$ whose heart $\sh_M$ is equivalent to $\mod{\End(M)}$.
\end{enumerate}
\end{proposition}

We denote the class of large silting objects in $\sD(\Lambda)$, up to equivalence, by $\SiltLambda$. Similarly, the class of classic silting objects in $\per{\Lambda}$, up to equivalence, is denoted by $\siltLambda$. As discussed in Remark \ref{comparison}, there is an embedding of $\siltLambda$ into $\SiltLambda$ and, by abuse of notation, we shall write $\siltLambda \subseteq \SiltLambda$.
For a classic silting object $M$ in $\per{\Lambda}$ and $n \ge 1$, the partial order in Definition~\ref{partial-order} defines the following subclasses of $\SiltLambda$:
\begin{align*}
\SiltL{n+1}{M} & \coloneqq \{ N \in \SiltLambda \mid \Add{M} \geq \Add{N} \geq \Add{M[n]}\}; \\
\siltL{n+1}{M} & \coloneqq \SiltL{n+1}{M} \cap \siltLambda = \{ N \in \siltLambda \mid \add{M} \geq \add{N} \geq \add{M[n]}\}.
\end{align*}

\subsection{Silting modules, $\tau$-tilting finiteness and silting discreteness}

Silting modules were introduced in \cite{AMV16} as infinite-dimensional analogues of support $\tau$-tilting modules.  For the original definition of support $\tau$-tilting module we refer to \cite{AIR14}.

\begin{definition}[{\cite[Definition 3.7]{AMV16}}]
A $\Lambda$-module $M$ is a \emph{silting module} if there is an exact sequence
\[
\begin{tikzcd}
P \ar[r, "\sigma"] & Q \ar[r] & M \ar[r] & 0
\end{tikzcd}
\]
with $P$ and $Q$ projective $\Lambda$-modules such that the class 
\[
\cD_\sigma \coloneqq \{X \in \Mod{\Lambda} \mid \Hom_R(\sigma,X) \text{\ is\ an\ epimorphism}\}
\]
coincides with  the class $\Gen{M}$ of modules which are epimorphic images of coproducts of $M$. 
Two silting modules $M$ and $N$ are said to be \emph{equivalent} if $\Add{M}=\Add{N}$.
\end{definition}

Over a finite-dimensional algebra, a module is  support $\tau$-tilting  if and only if it a finite-dimensional silting module \cite[Proposition 3.15]{AMV16}. A finite-dimensional algebra is \emph{$\tau$-tilting-finite} if it has only finitely many support $\tau$-tilting modules up to equivalence \cite{DIJ19}. It turns out that these are precisely the algebras whose silting modules coincide, up to equivalence, with the support $\tau$-tilting modules.

\begin{theorem}[{\cite[Theorem 4.8]{AMV19}}] \label{taufinite}
The following are equivalent for a finite-dimensional $\kk$-algebra $\Lambda$.
\begin{enumerate}
\item $\Lambda$ is $\tau$-tilting-finite. 
\item Every silting $\Lambda$-module is finite dimensional up to equivalence.
\item Every torsion pair in $\mathsf{Mod}(\Lambda)$ is of the form $\mathsf{Gen}(T)$ for a finite-dimensional silting module $T$.
\end{enumerate}
\end{theorem}

The triangulated category analogue of $\tau$-tilting finiteness is \emph{silting discreteness} \cite{AM17}. We recall the following characterisation of a silting-discrete finite-dimensional algebra.

\begin{theorem} \label{silting-discrete}
 The following statements are equivalent for  $\Lambda$.
\begin{enumerate}
\item For any  $M$ in $\siltLambda$ and any $n>1$, the set $\siltL{n}{M}$ is finite.
\item For any $M$ in $\siltLambda$, the set $\siltL{2}{M}$ is finite.
\item For any $M$ in $\siltLambda$, the finite-dimensional algebra $\End(M)$ is $\tau$-tilting finite.
\end{enumerate}
\end{theorem}

If these equivalent conditions hold then $\Lambda$ is called \emph{silting discrete}.

\begin{proof}
The assertion $(1)\Leftrightarrow (2)$ is \cite[Theorem 2.4]{AM17}, and the assertion $(2)\Leftrightarrow (3)$ is \cite[Theorem 4.6]{IJY14}. 
\end{proof} 

\subsection{Another characterisation of silting discreteness}

In this section, we add a further characterisation to the list in Theorem~\ref{silting-discrete} by 
proving that silting-discrete finite-dimensional algebras are those whose large, bounded silting theory in $\sD(\Lambda)$ coincides with the classic silting theory in $\per{\Lambda}$. This can be regarded as a triangulated version of Theorem~\ref{taufinite}.

\begin{theorem}\label{large is small silting discrete}
A finite-dimensional algebra $\Lambda$ is silting discrete if and only if every large silting object in $\sD(\Lambda)$ which lies in $\Kb(\Proj{\Lambda})$ is perfect up to equivalence. 
In other words, $\Lambda$ is silting discrete if and only if  $\SiltL{n}{\Lambda} = \siltL{n}{\Lambda}$ for each $n>1$. 
\end{theorem}

We recover the following result of Aihara and Mizuno immediately from Theorem~\ref{large is small silting discrete}.

\begin{corollary}[{\cite[Theorem 2.15]{AM17}}]
Let $\Lambda$ be a silting-discrete finite-dimensional algebra. Then every classic presilting object $X$ in $\per{\Lambda}$ admits a complement in $\per{\Lambda}$.
\end{corollary}

\begin{proof}
It follows from 
Remark \ref{small admits large complement} that $X$ admits a complement to a large silting object in $\sD(\Lambda)$, i.e.\ there is $V$ such that $X\oplus V$ is large silting in $\sD(\Lambda)$, and moreover, 
$V$ can be chosen in $\Kb(\Proj{\Lambda})$. Therefore  $X\oplus V$  is a large silting object that is a bounded complex of projective $\Lambda$-modules, and by Theorem~\ref{large is small silting discrete}, it is equivalent to a classic silting object in $\per{\Lambda}$.
\end{proof}

The rest of this section is devoted to the proof of Theorem~\ref{large is small silting discrete}.
For the reverse implication, 
we will need the following generalisation of \cite[Theorem 3.2]{AIR14} and \cite[Theorem 4.9]{AMV16} which follows the spirit of \cite{IJY14} in making a `basis-free' statement. 

\begin{proposition} \label{silting modules bijection}
Let $M$  a classic silting object in $\per{\Lambda}$. Write $\Gamma=\End(M)$ and let $H^0:=H^0_M \colon \sD(\Lambda) \to \Mod{\Gamma}$ be the cohomological functor associated to the t-structure $(\sV_M,\sW_M)$ according to Proposition~\ref{Koenig-Yang}.
There is a bijection 
\begin{align*}
\SiltL{2}{M} & \bij \{\text{silting $\Gamma$-modules up to equivalence}\} \\
T & \longmapsto H^0_M(T) 
\end{align*}
which restricts to a bijection  
\[
\siltL{2}{M} \bij \{\text{support $\tau$-tilting $\Gamma$-modules up to equivalence}\}.
\] 
\end{proposition}

\begin{proof}
We fix the notation $\sU_T:={}^\perp \sV_T$ and $\sU_M:={}^\perp\sV_M$ for the left orthogonal subcategories of $\sV_T$ and $\sV_M$ in $\sD(\Lambda)$, as in previous sections.

We begin by showing the assignment is well defined. 
Suppose $T$ is an object in $\sD(\Lambda)$ that lies in $\SiltL{2}{M}$. By Lemma~\ref{interm}(2) there is a triangle of the form
\[
\begin{tikzcd}
M_1 \ar[r, "\Sigma"] & M_0 \ar[r] & T \ar[r] & M_1[1],
\end{tikzcd}
\]
with $M_0$ and $M_1$ in $\Add{M}$. By Proposition~\ref{Koenig-Yang}(2), applying the cohomological functor $H^0$ to this triangle  we obtain a projective presentation of $H^0(T)$:
\[
\begin{tikzcd}
H^0(M_1) \ar[r, "\sigma:=H^0(\Sigma)"] &[4ex] H^0(M_0) \ar[r] & H^0(T) \ar[r] & 0.
\end{tikzcd}
\]
We claim that $H^0(T)$ is a silting $\Gamma$-module with respect to the projective presentation $\sigma$.  

{\bf Step 1:} {\it A $\Gamma$-module $X$, regarded as an object of $\sH_M$, lies in $\cD_\sigma$ if and only if it lies in $\sV_T$, or equivalently, $\Hom_{\sD(\Lambda)}(T,X[1])=0$.} 

The canonical maps $h_1 \colon M_1 \to H^0(M_1)$ and $h_2 \colon M_0 \to H^0(M_0)$ induce the following commutative diagram.
\[
\begin{tikzcd}[column sep= 15ex]
\Hom_{\sD(\Lambda)}(H^0(M_0),X) \ar[r, "{\Hom_{\sD(\Lambda)}(\sigma,X)}"] \ar{d}{\sim}[swap]{\Hom_{\sD(\Lambda)}(h_0,X)} & \Hom_{\sD(\Lambda)}(H^0(M_1),X) \ar{d}{\Hom_{\sD(\Lambda)}(h_1,X)}[swap]{\sim} \\ 
\Hom_{\sD(\Lambda)}(M_0,X) \ar[swap,r, "{\Hom_{\sD(\Lambda)}(\Sigma,X)}"] & \Hom_{\sD(\Lambda)}(M_1,X)
\end{tikzcd}
\]
From this, we conclude that $X$ lies in $\cD_\sigma$ if and only if $\Hom_{\sD(\Lambda)}(\Sigma,X)$ is surjective. This occurs if and only if $\Hom_{\sD(\Lambda)}(T,X[1])=0$ as  $\Hom_{\sD(\Lambda)}(M_0,X[1])=0$.
Moreover, this happens if and only if $X$ lies in $\sV_T$. Indeed, $X$ already lies in $V_T[-1]$ as $\Hom_{\sD(\Lambda)}(T,X[i])=0$ for all $i>1$ because $T$ lies in $\sU_T[1]\subseteq\sU_M[2]$. This latter claim follows from the assumption $\sV_M[1]\subseteq\sV_T$, which gives $\sU_T\subseteq\sU_M[1]$.

{\bf Step 2:} {\it We have $\Gen{H^0(T)} \subseteq \cD_\sigma$.}

It suffices to show that ${H^0(T)^{(I)}} $ lies in $ \cD_\sigma$ for any set $I$, because $\cD_\sigma$ is closed under quotients. Note that $H^0(T)^{(I)}\cong H^0(T^{(I)})$ because $M$ is compact. 
As $T$ and $T^{(I)}$ lie in $\sV_M$, truncating  with respect to $(\sV_M[1],\sW_M[1])$ gives a triangle
\[
\begin{tikzcd}
V[1] \ar[r] & T^{(I)} \ar[r] & W[1] = H^0(T^{(I)}) \ar[r] & V[2].
\end{tikzcd}
\]
As $\sV_M[1]\subseteq \sV_T$ and $T^{(I)} \in \sV_T$ because $T$ is large presilting, it follows that $H^0(T^{(I)})$ lies in $\sV_T$, and thus in $ \cD_\sigma$  by Step 1.

{\bf Step 3:} {\it We have $\cD_\sigma \subseteq \Gen{H^0(T)}$.}

Let $X$ be an object in $\cD_\sigma$ and take the universal map $u\colon H^0(T)^{(I)} \to X$, where $I$ is a basis for the $\kk$-vector space $\Hom_\Gamma(H^0(T),X)$. In order to prove that $u$ is an epimorphism in $\sH_M$, we use the fact that $H^0(T)^{(I)}\cong H^0(T^{(I)})$ again and consider the triangle
\[
\begin{tikzcd}
H^0(T^{(I)}) \ar[r,"u"] & X \ar[r] & K \ar[r] & H^0(T^{(I)})[1]. 
\end{tikzcd}
\]
We show that  $H^0(K)=0$. Since $H^0(M)$ is a projective generator of $\sH_M$ by Proposition~\ref{Koenig-Yang}\eqref{projective-generator}, this amounts to showing that  $\Hom_{\sD(\Lambda)}(M,K)=0$. Now, by assumption $\Add{T}[-1]\ge \Add{M}\ge \Add{T}$, so we know from Lemma~\ref{interm}(2) that    $M$ lies in $\Add{T[-1]} \ast \Add{T}$. Thus, it suffices to check that 
\[\Hom_{\sD(\Lambda)}(T,K)=0=\Hom_{\sD(\Lambda)}(T,K[1]).\]

Since $X$ is an object of $\sH_M$, any morphism $T \to X$ factors through the canonical map $T \to H^0(T)$. This shows that $\Hom_\sT(T,u)$ is an epimorphism. 
It follows that $\Hom_{\sD(\Lambda)}(T,K)=0$ since $\Hom_{\sD(\Lambda)}(T,H^0(T)^{(I)}[1])=0$ from Step 2. Moreover, as $X$ lies in $\cD_\sigma$, we have from Step 1 that $\Hom_{\sD(\Lambda)}(T,X[1])=0$, and we conclude that $\Hom_{\sD(\Lambda)}(T,K[1])=0$. 
Hence we have that $X$ lies in $\Gen{H^0(T)}$ as claimed.

We have thus shown that the assignment $T \mapsto H^0(T)$ is well defined. In order to prove the bijectivity of this map, we observe the following.

{\bf Step 4:} {\it We have $\sV_T=\sV_M[1] * \Gen{H^0(T)}$.}

Indeed, we have that $\sV_M[1] * \Gen{H^0(T)} \subseteq \sV_T$ by assumption and Step 1. For the reverse inclusion, truncate an object $X$ in $\sV_T$ with respect to the t-structure $(\sV_M[1], \sW_M[1])$
\[
\begin{tikzcd}
V[1] \ar[r] & X \ar[r] & W[1] \ar[r] & V[2],
\end{tikzcd}
\]
where, again, $W[1]=H^0(X)$ lies in $\sH_M$ because $\sV_T$ is contained in $\sV_M$.
Now, $H^0(X)$ lies in $\sV_T$ since $\sV_M[1] \subseteq \sV_T$, whence $H^0(X)$ lies in $\cD_\sigma = \Gen{H^0(T)}$ by Step 1.

{\bf Step 5:} {\it The assignment $T \mapsto H^0(T)$ is injective.}

Suppose $T_1$ and $T_2$ are objects of $\SiltL{2}{M}$ such that $\Add{H^0(T_1)} = \Add{H^0(T_2)}$. Then $\Gen{H^0(T_1)} = \Gen{H^0(T_2)}$, and by Step 4 we have $\sV_{T_1} = \sV_{T_2}$. Now we use Theorem~\ref{large correspondence} asserting that  a silting object is determined by the co-t-structure up to equivalence.

{\bf Step 6:} {\it The assignment $T \mapsto H^0(T)$ is surjective.}

Suppose $Y$ is a silting $\Gamma$-module with respect to a map $\sigma \colon H^0(M_1)\to H^0(M_0)$. By Proposition~\ref{Koenig-Yang}(2), $H^0|_{\Add{M}} \colon \Add{M} \to \Proj{\sH_M} = \Add{H^0(M)}$ is an equivalence, and so, there is a unique map $\Sigma\colon M_1 \to M_0$ such that $H^0(\Sigma)=\sigma$. 
Thus, we set $T$ to be the cone of $\Sigma$, i.e. we consider the triangle
\[
\begin{tikzcd}
M_1 \ar[r,"\Sigma"] & M_0 \ar[r] & T \ar[r] & M_1[1],
\end{tikzcd}
\]
from which we observe that $H^0(T)=Y$ and that $T$ lies in $\Add{M}\ast\Add{M[1]}$. We check that $T$ is a large silting object using Theorem~\ref{Breaz}. By the construction of $T$ it is clear that $T$ lies in $\sV_M$ and that $\sV_M[1]\subseteq \sV_T$. It remains to see that $\mathsf{Add}(T)$ is contained in $\sV_T$ and that $T$ is a weak generator.

We first show that $\mathsf{Add}(T)$ is contained in $\sV_T$. For a $\Gamma$-module $X$, applying $\Hom_{\sD(\Lambda)}(-,X)$ to the triangle above tells us that $\Hom_{\sD(\Lambda)}(T,X[1])=0$ if and only if $\Hom_{\sD(\Lambda)}(\Sigma, X)$ is surjective. Consider the commutative diagram:
\begin{equation} \label{surjective}
\tag{$*$} \begin{tikzcd}[column sep= 15ex]
\Hom_{\sD(\Lambda)}(H^0(M_0),Y^{(I)}) \ar[r, "{\Hom_{\sD(\Lambda)}(\sigma,Y^{(I)})}"] \ar{d}{\sim}[swap]{\Hom_{\sD(\Lambda)}(h_0,Y^{(I)})} & \Hom_{\sD(\Lambda)}(H^0(M_1),Y^{(I)}) \ar{d}{\Hom_{\sD(\Lambda)}(h_1,Y^{(I)})}[swap]{\sim} \\ 
\Hom_{\sD(\Lambda)}(M_0,Y^{(I)}) \ar[swap,r, "{\Hom_{\sD(\Lambda)}(\Sigma,Y^{(I)})}"] & \Hom_{\sD(\Lambda)}(M_1,Y^{(I)})\\
\Hom_{\sD(\Lambda)}(M_0,T^{(I)}) \ar[swap,r, "{\Hom_{\sD(\Lambda)}(\Sigma,T^{(I)})}"]\ar{u}{\Hom_{\sD(\Lambda)}(M_0,h_{T^{(I)}})}[swap]{\sim} & \Hom_{\sD(\Lambda)}(M_1,T^{(I)})\ar{u}{\sim}[swap]{\Hom_{\sD(\Lambda)}(M_1,h_{T^{(I)}})}
\end{tikzcd}
\end{equation}
where $h_{T^{(I)}}\colon T^{(I)}\to H^0(T^{(I)})\cong H^0(T)^{(I)}=Y^{(I)}$, $h_1\colon M_1\to H^0(M_1)$ and $h_0\colon M_0\to H^0(M_0)$ are the canonical maps coming from the fact that $M_1$, $M_0$ and $T^{(I)}$ lie in $\sV_M$. The top vertical maps are isomorphisms as in Step 1, and the bottom vertical maps are isomorphisms because $M$ is silting and $T$ lies in $\sV_M$. Since, by assumption, $Y$ is a silting module with respect to $\sigma$ it follows that $\Hom_{\sD(\Lambda)}(\Sigma,T^{(I)})$ is surjective, as required.

Next, we argue that $T$ is a weak generator exactly as in the proof of (4)$\Rightarrow$(1) in \cite[Theorem 4.9]{AMV16}; we transcribe the proof to our setting and notation for the convenience of the reader. Let $Z$ be an object in $\sT$ for which $\Hom_{\sD(\Lambda)}(T,Z[j])=0$ for all $j$ in $\bZ$. For $i$ in $\bZ$, let $v_i$ denote the right adjoint of the inclusion of $\sV_M[i]$ into $\sT$ (i.e., $v_i$ is the truncation with respect to $\sV_M[i]$). Since $T$ lies in $\sV_M$, it follows that $\Hom_{\sD(\Lambda)}(T,-)\cong \Hom_{\sD(\Lambda)}(T,v_0(-))$ and, as $\sV_M[1] \subseteq \sV_T$, there is a natural epimorphism 
\[
\begin{tikzcd}
\Hom_{\sD(\Lambda)}(T,v_0(-)) \arrow[two heads]{r} & \Hom_{\sD(\Lambda)}(T,H^0(-))\cong \Hom_{\sD(\Lambda)}(Y,H^0(-)).
\end{tikzcd}
\]
Hence, $H^j(Z)$ lies in the torsionfree class $Y\orth$ in $\sH_M$ for all $j$ in $\bZ$. From the triangle
\[
H^0(Z[j+1])[-1] \longrightarrow v_1(Z[j+1]) \longrightarrow v_0(Z[j+1]) \longrightarrow H^0(Z[j+1])
\]
we deduce that $\Hom_{\sD(\Lambda)}(T,v_1(Z[j+1]))=0$. Notice that $v_1(Z[j+1])\cong v_0(Z[j])[1]$, and thus $v_0(Z[j])$ lies in $\sV_T$ for all $j$ in $\bZ$. This implies that $\Hom_{\sD(\Lambda)}(\Sigma,v_0(Z[j]))$ is surjective and, using a diagram as in \eqref{surjective} above, that $\Hom_{\sD(\Lambda)}(\sigma,H^0(Z[j]))$ is surjective. This shows that $H^j(Z)$ lies in $\cD_\sigma$ for all $j$ in $\bZ$. Since $(\cD_\sigma,Y\orth)$ is a torsion pair in $\sH_M$, we conclude that $H^j(Z)=0$ for all $j$ and, since $(\sV_M,\sW_M)$ is nondegenerate, we conclude that $Z=0$. This concludes the proof that $T$ is indeed a large silting object in $\sD(\Lambda)$.

Thus,  $T \mapsto H^0(T)$ is a bijection between silting $\Gamma$-modules and $\SiltL{2}{M}$. 

{\bf Step 7:} {\it The assignment $T \mapsto H^0(T)$ restricts to bijection $\siltL{2}{M}$ and support $\tau$-tilting $\Gamma$-modules.}

This bijection is well known, see \cite[Theorem 4.6]{IJY14}. However, for completeness, we explain how the statement can be recovered as a restriction of the assignment in the cocomplete case above.

If $T$ is compact, then $T$ is an object in $\add{M} * \add{M[1]}$ and, therefore, $H^0(T)$ is a finite-dimensional $\Gamma$-module. Conversely, suppose  $H^0(T)$ is a finite-dimensional silting $\Gamma$-module witnessed by a projective presentation $\sigma$ in $\proj{\sH_M} = \add{H^0(M)}$. As in Step 6, we can lift $\sigma$ to a map $\Sigma$ in $\add{M}$. We then observe that the cone of $\Sigma$ is a compact silting object lying in $\siltL{2}{M}$ and whose zeroth cohomology is $H^0(T)$. 
\end{proof}

We now turn to the forward implication in Theorem~\ref{large is small silting discrete}, which is implicitly contained in \cite[Lemma 3.5]{PSZ18}. We provide details for the convenience of the reader.

\begin{lemma} \label{large and bounded = small}
Let $\Lambda$ be a silting-discrete, finite-dimensional $\kk$-algebra. Suppose $M$ and $S$ are large silting objects such that $S$ lies in $\per{\Lambda}$ and $M$ lies in $\SiltL{n}{S}$ for some natural number $n \geq 1$. Then there exists a large silting object $T$ in $\per{\Lambda}$ such that $M$ lies in $\SiltL{2}{T}$. In particular, $\SiltL{n}{\Lambda} = \siltL{n}{\Lambda}$ for all $n > 0$.
\end{lemma}

\begin{proof}
As $M$ lies in $\SiltL{n}{S}$, we have $\sV_S[n] \subseteq \sV_M \subseteq \sV_S$, or equivalently, $\sW_S[n] \supseteq \sW_M \supseteq \sW_S$.
By Proposition~\ref{Koenig-Yang}(3), the pair
$(\sv_S,\sw_S) = (\sV_S \cap \Db(\Lambda), \sW_S \cap \Db(\Lambda))$ is a t-structure in $\Db(\Lambda)$ with heart $\sh_S \simeq \mod{\Gamma}$, where $\Gamma := \End_{\sD(\Lambda)}(S)$.
Applying Lemma~\ref{saorin-stovicek} on the set $\{S[k+1], M[k]\mid k\ge 0\}$, we obtain a t-structure
\[
(\sV, \sW:= \sW_S[1] \cap \sW_M\big) \text{ in } \sD(\Lambda).
\]
One can check that 
\begin{equation}\label{first set}\sW[n-1] \supseteq \sW_M \supseteq \sW,\end{equation} 
\begin{equation}\label{second set}\sW_S[1] \supseteq \sW \supseteq \sW_S.\end{equation} 
It suffices to find a large silting object $N \in \per{\Lambda}$ such that $(\sV, \sW) = (\sV_N, \sW_N).$ Indeed, $M$  then lies in $\SiltL{n-1}{N}$ by (\ref{first set}), and the result  follows by induction.

The inclusions (\ref{second set}) show that the t-structure $(\sV , \sW)$ is a Happel--Reiten--Smal\o\ tilt of $(\sV_S, \sW_S)$ (see \cite[Proposition 2.1]{HRS96}), i.e.~there is a torsion pair $(\cT,\cF)$ in $\sH_S$ such that $(\sV , \sW) = (\sV_S[1] * \cT, \cF * \sW_S)$, see for example \cite[Proposition 2.1]{Woolf10}. Since $\Lambda$ is silting discrete, it follows from Theorem \ref{silting-discrete} that $\Gamma$ is $\tau$-tilting finite and, therefore, by Theorem \ref{taufinite}, we have that $\cT=\mathsf{Gen}(Y)$, for a finite-dimensional silting $\Gamma$-module. By Proposition \ref{silting modules bijection}, $Y=H^0(N)$ for some silting object $N$ in $\mathsf{silt}^2_S(\Lambda)$ and, by Step 4 of proof of Proposition \ref{silting modules bijection}, it follows that $\sV_N=\sV_S[1] * \cT=\sV$, as wanted. 

For the last statement, pick an object  $T$ in $\SiltL{n}{\Lambda}$. It lies in $\SiltL{2}{M}$ for some classic silting object in $\per{\Lambda}$ and corresponds to some silting $\End(M)$-module under the bijection in  Proposition~\ref{silting modules bijection}. By assumption $\End(M)$ is $\tau$-tilting finite, and by Theorem~\ref{taufinite} we obtain that $T$ lies in $\siltL{2}{M}$. So  $T$ is a perfect complex and lies in $\siltL{n}{\Lambda}$.
\end{proof}

\begin{proof}[Proof of Theorem~\ref{large is small silting discrete}]
One implication has just been proven in  Lemma~\ref{large and bounded = small}.
For the other implication, suppose $\SiltL{n}{\Lambda} = \siltL{n}{\Lambda}$ for all $n > 1$. 
Given a classic  silting object $M$ in $\per{\Lambda}$, it follows that $\SiltL{2}{M} = \siltL{2}{M}$. 
Indeed, as $M \in \per{\Lambda}$ there exists $n > 1$ such that $M \in \SiltL{n}{\Lambda}$. Hence, for $N \in \SiltL{2}{M}$, we have $N \in \SiltL{n+1}{\Lambda} = \siltL{n+1}{\Lambda}$ and it follows that $N \in \siltL{2}{M}$.
By Proposition~\ref{silting modules bijection}, this can be rephrased as a property of the algebra $\Gamma=\End(M)$, namely,  all silting $\Gamma$-modules are  finite dimensional up to equivalence. But this means that  $\Gamma$  is $\tau$-tilting finite by Theorem~\ref{taufinite}. By Theorem~\ref{silting-discrete} we conclude that $\Lambda$ is silting discrete. 
\end{proof}

%============================================================================
% ACKNOWLEDGMENTS
%============================================================================

\subsection*{Acknowledgments}

We would like to thank the referee for a careful reading of the article and useful suggestions.
The first and third named authors are members of the network INdAM-G.N.S.A.G.A and acknowledge support from the project \textit{REDCOM: Reducing complexity in algebra, logic, combinatorics}, financed by the programme  \textit{Ricerca Scientifica di Eccellenza 2018} of the Fondazione Cariverona, and from the project funded by  NextGenerationEU under NRRP, Call PRIN 2022  No.~104 of February 2, 2022 of Italian Ministry of University and Research; Project 2022S97PMY \textit{Structures for Quivers, Algebras and Representations (SQUARE)}. 
The second author acknowledges the support of the EPSRC of the United Kingdom via grant no. EP/V050524/1.
The third author additionally acknowledges financial support from the Department of Mathematics `Tullio Levi-Civita' through its BIRD - Budget Integrato per la Ricerca dei Dipartimenti 2022, within the project \textit{Representations of quivers with commutative coefficients}.

%============================================================================
% REFERENCES
%============================================================================

\vspace{-0.4cm}

\Addresses
%============================================================================

\begin{landscape}

\appendix

%\section{Example in $\Db(\kk \widetilde{A}_2)$}

\begin{figure} 
\begin{center}
\input{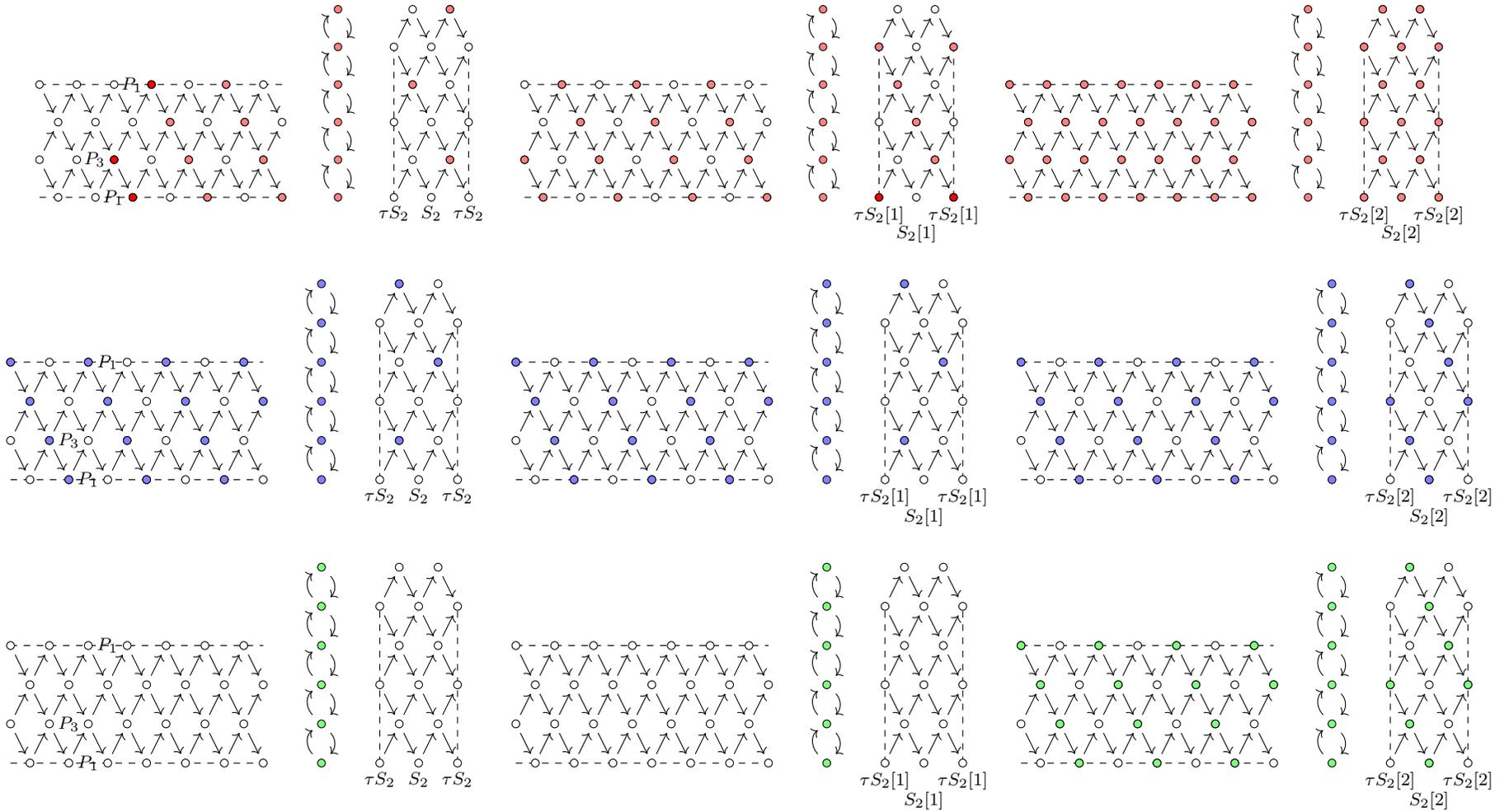}
\end{center}
\caption{Each figure shows a region of the Auslander--Reiten quiver of $\Db(\kk Q)$. Top: the coaisle $\sV_M$ associated to the silting object $M$ is marked in red, with the silting object $M$ marked in deeper red. 
Middle: the (unbounded) coaisle $\sV_X$ associated with the presilting object $X = S_2[2] \in \sV_M$ is marked in blue. 
Bottom: the intersection $\sV \coloneqq \sV_X \cap \sV_M  = \sV_{X \oplus M}$. One can see that no object in the transjective component containing the shifted projective objects admits a left $\sV$-approximation. $\sV$ is not the coaisle of a co-t-structure.} \label{A2-example}
\end{figure}
\end{landscape}

\newpage

\end{document}